\documentclass[a4paper]{amsart}
\usepackage{graphicx}

\newtheorem{theorem}{Theorem}[section]
\newtheorem{lemma}[theorem]{Lemma}
\newtheorem{proposition}[theorem]{Proposition}
\theoremstyle{definition}
\newtheorem{definition}[theorem]{Definition}

\newcommand{\breadth}{\operatorname{breadth}}
\begin{document}

\title[MUTATION INVARIANCE OF THE ARC INDEX 
FOR SOME MONTESINOS KNOTS]{MUTATION INVARIANCE OF THE ARC INDEX  \\
FOR SOME MONTESINOS KNOTS}

\author{GYO TAEK JIN $\dagger$ and HO LEE $\ddagger$}

\address{Deparment of Mathematical Sciences, KAIST, Daejeon 34141, Korea \newline $\dagger$ trefoil@kaist.ac.kr \\ $\ddagger$  figure8@kaist.ac.kr}

\begin{abstract}
For the alternating knots or links, mutations do not change the arc index. In the case of nonalternating knots, some  semi-alternating knots or links have this property. We mainly focus on the problem of mutation invariance of the arc index for nonalternating knots which are not semi-alternating. In this paper, we found families of infinitely many mutant pairs/triples of Montesinos knots with the same arc index.
\end{abstract}


\maketitle

\section{Introduction}	
There are many different ways to describe an arc presentation~\cite{Cromwell-book}. In this paper, we use the following definition~\cite{1}. Let $D$ be a diagram of a link $L$. Assume that there is a simple closed curve $C$ meeting $D$ in $n$ distinct points which divide $D$ into $n$ arcs $\alpha_1, \cdots, \alpha_n$ with the following rules:

\begin{enumerate}
\item Each $\alpha_i$ has no self crossing.
\item If $\alpha_i$ crosses over $\alpha_j$ at a crossing then $i>j$ and it crosses over $\alpha_j$ at any other crossings with $\alpha_j$.
\item For each $i$, there is an embedded disk $d_i$ such that $\partial d_i = C$ and $\alpha_i \subset d_i$.
\item For distinct $i$ and $j$, $d_i \cap d_j = C$.
\end{enumerate}

\begin{definition}
The pair $(D,C)$ which satisfies above properties is called an \emph{arc presentation} of $L$ with $n$ \emph{arcs} and $C$ is called the \emph{binding circle} of the arc presentation. Given a link $L$, the minimal number of arcs among all arc presentations of $L$ is called the \emph{arc index} of $L$ and is denoted by $\alpha(L)$.
\end{definition}

Cromwell showed that every link diagram is isotopic to a grid diagram and thus has an arc presentation~\cite{2}.
By removing a point $P$ from $C$ away from $L$, we may identify $C \backslash P$ with the $z$-axis and each $d_i \backslash P$ with a vertical half plane along the $z$-axis. This shows that an arc presentation is equivalent to an `open-book decomposition' as shown in Figure~\ref{fig1}. 
\begin{figure}[th]
  \centering
    \includegraphics[height=3cm]{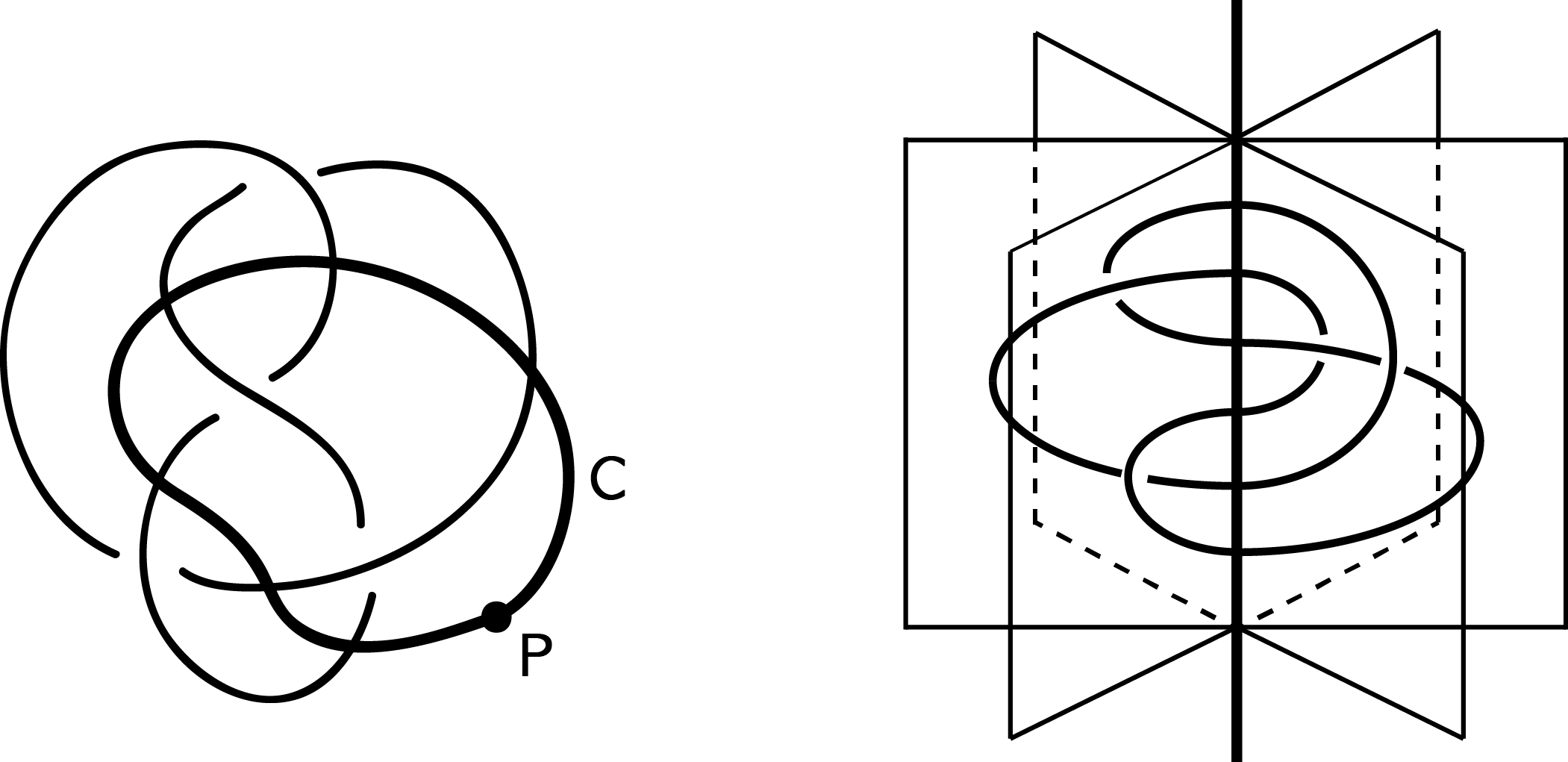}
\vspace*{8pt}
  \caption{An arc presentation of the figure eight knot}\label{fig1}
\end{figure}

\emph{Mutation} is an operation on a knot  which may change the knot type. The most commonly used description of mutation is combinatorial, arising from Conway's definition of a tangle~\cite{Conway-1970}. Suppose that a knot $K$ can be decomposed into two  2-string tangles. Suppose that one of them has four ends at NW, NE, SE and SW of a circle around it. We rotate this tangle about the east-west axis, the north-south axis, or the vertical axis through $180^\circ$, as shown in  Figure~\ref{fig2}. A new diagram thus obtained corresponds to a possibly new knot, called a \emph{mutant} of the given one. 

\begin{figure}[th]
  \centering
		\includegraphics[height=3cm]{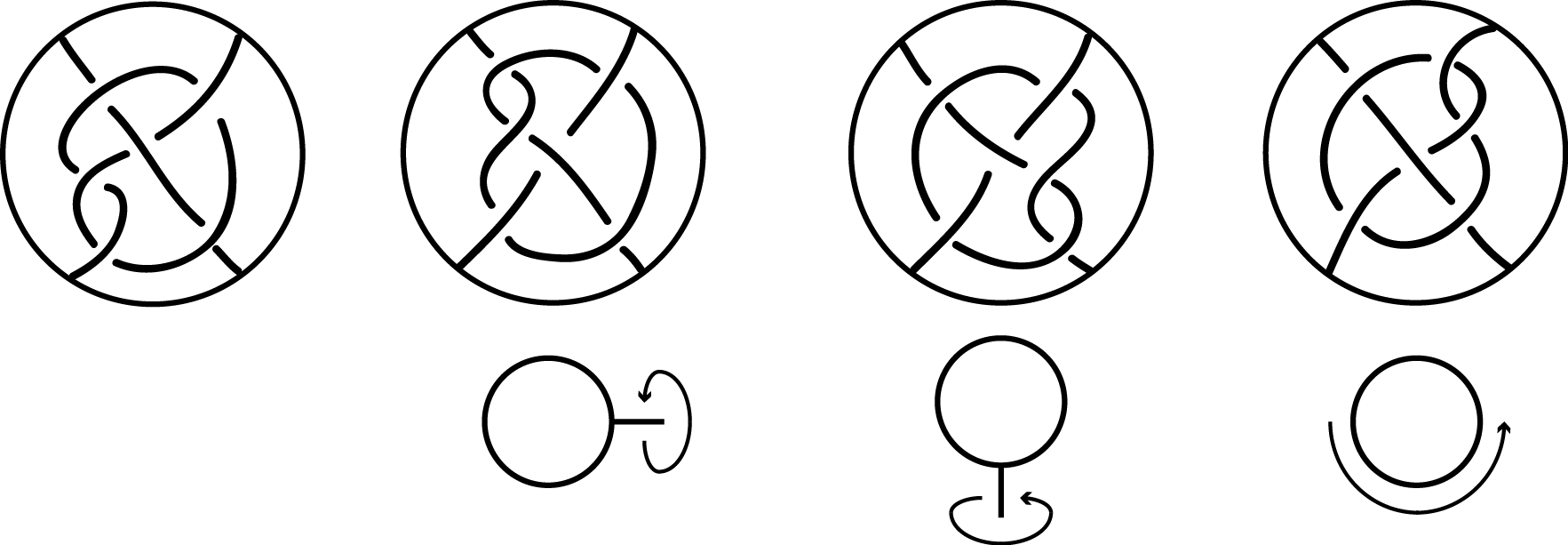}
\vspace*{8pt}
  \caption{Mutations}\label{fig2}
\end{figure}

\begin{figure}[h]
  \centering
		\includegraphics[height=2.5cm]{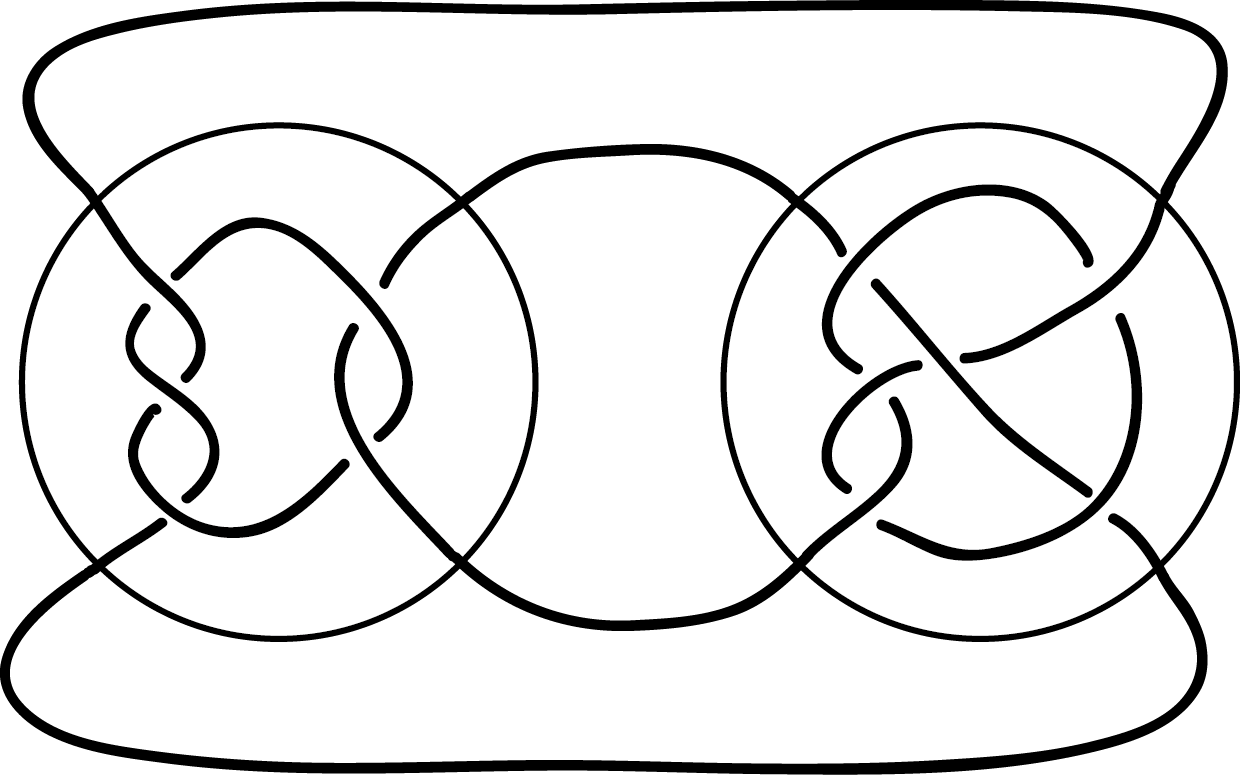}
\hspace{10mm}
		\includegraphics[height=2.5cm]{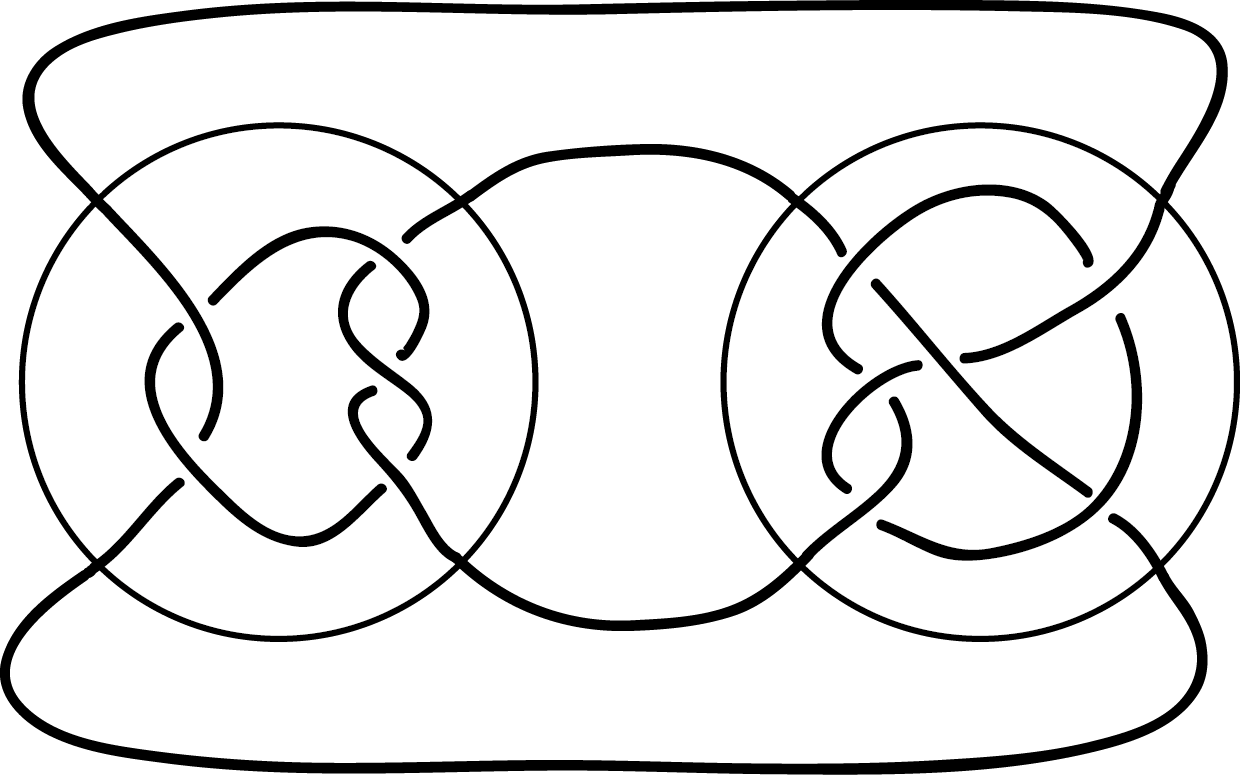}
  \caption{The Kinoshita-Teresaka knot and the Conway knot}\label{fig3}
\end{figure}

Up to 10 crossings, there are no inequivalent mutant knots~\cite{4}. The Kinoshita-Teresaka knot $(11n42)$ and the Conway knot  $(11n34)$ are a well-known example of inequivalent mutant pair.
As mutant knots share polynomial invariants such as Alexander polynomial, Jones polynomial, and Kauffman polynomial, they are  difficult to be distinguished. Stoimenow identified the mutant groups of prime knots up to 15 crossings and proved that crossing number is invariant under mutation for prime knots  up to 15 crossings~\cite{4}. But it is not known whether diagrammatic invariants such as the crossing number, the braid index, and the arc index are unchanged by mutation.

\section{Lower and Upper Bounds}
The Kauffman polynomial $F_L(a,z)$ of an oriented link $L$ is defined by 
$$F_L(a,z) = a^{- w(D)} \Lambda_{D}(a,z)$$
where $D$ is a diagram of $L$, $w(D)$ is the writhe of $D$ and $\Lambda_{D}(a,z)$ is the polynomial in $a$ and $z$ determined by the following properties:

\begin{enumerate}
\item $\Lambda_{O} (a,z) = 1$ where $O$ is the trivial knot diagram.
\item For any four diagrams $D_{+}, D_{-}, D_{\|}, D_{=}$ which  differ inside a disc as shown below but are identical outside,
we have the relation 
$$\Lambda_{D_+} (a,z) + \Lambda_{D_-} (a,z)=  z (\Lambda_{D_{\|}} (a,z) + \Lambda_{D_=} (a,z))$$
\item For any three diagrams $D_{\oplus}, D_{\ominus}, D$ which  differ inside a disc as shown below but are identical outside,
 we have the relation 
 $$a\Lambda_{D_\oplus} (a,z) =  \Lambda_{D} (a,z) = a^{-1}\Lambda_{D_\ominus} (a,z)$$
\end{enumerate}
\begin{figure}[th]
  \centering
		\includegraphics[height=1.3cm]{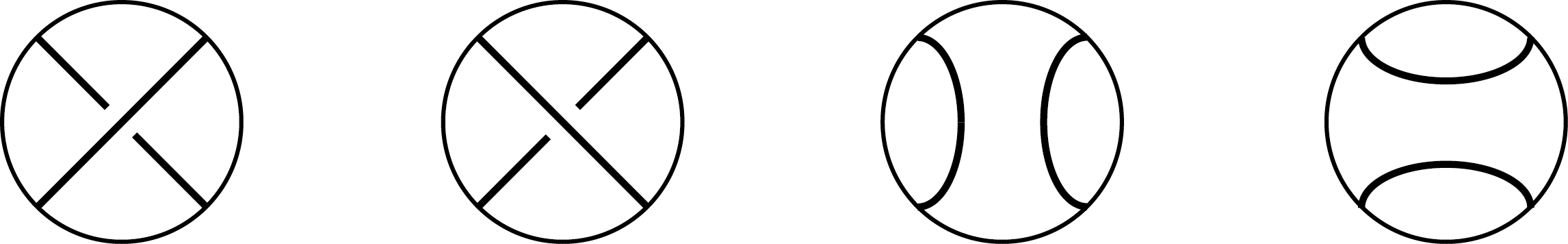}
		
$D_{+} \hspace{18mm} D_{-} \hspace{18mm} D_{\|} \hspace{18mm} D_{=}$

\medskip
		\includegraphics[height=1.3cm]{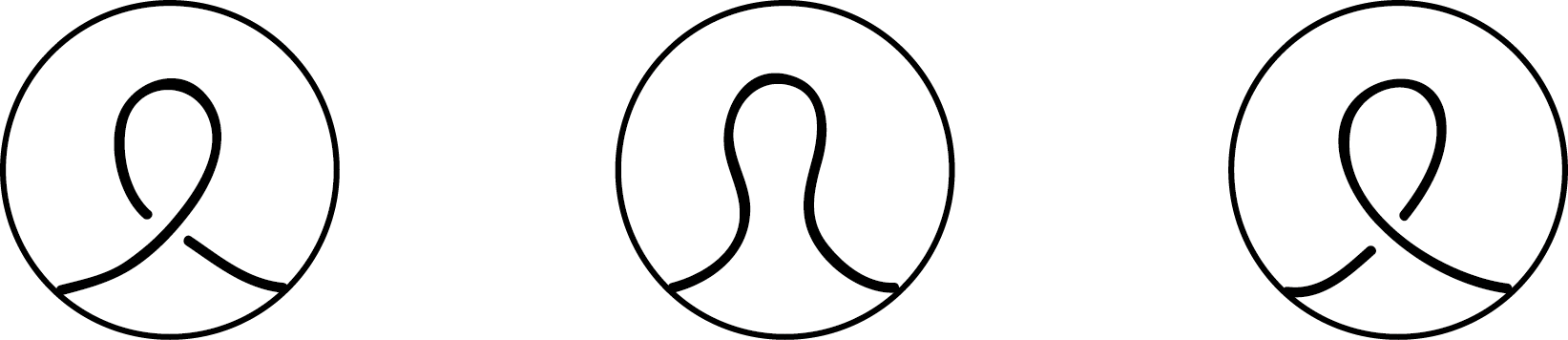}

$D_{\oplus} \hspace{19mm}  D \hspace{19mm} D_{\ominus}$
\end{figure}

For a link $L$, let $c(L)$ and $\breadth_a(F_{L})$ denote the minimal crossing number of $L$ and the Laurent degree in the variable $a$ of the Kauffman polynomial of $L$ (i.e. the difference between the hightest degree and the lowest degree of the variable $a$ in $F_L(a,z)$), respectively. We list some of the known results about the arc index and the minimal crossing number.

\begin{proposition}[\cite{2a,5a}]
If $L$ is an alternating link, $c(L) \le \breadth_a(F_L(a,z))$.
\end{proposition}

\begin{proposition}[\cite{6}]\label{prop:breadth+2}
For any link L, $\breadth_a(F_L(a,z))+2 \le \alpha(L)$.
\end{proposition}

\begin{proposition}[\cite{7}]
Let $L$ be any prime link. Then $\alpha (L) \le c(L) +2$. Moreover this inequality is strict if and only if $L$ is not alternating.
\end{proposition}

\begin{proposition}[\cite{jp}]
A prime link $L$ is nonalternating if and only if \break
$\alpha (L) \le c(L)$.
\end{proposition}

With the results above, it can be seen that for non-split alternating link $L$, the arc index is equal to the crossing number plus two. Since mutation does not change the alternating property, for non-split alternating knots, arc index is invariant under mutation.

In 1997, Beltrami and Cromwell constructed minimal arc presentations of some nonalternating links which are unaffected by mutations~\cite{8}.
A 2-string tangle is {\em strongly alternating} if both the numerator closure and the denominator closure are irreducible alternating diagrams. The nonalternating sum of two strongly alternating tangles is called a semi-alternating diagram. As usual, a knot is {\em semi-alternating} if it has a  semi-alternating diagram. For example, The Kinoshita-Teresaka knot and the Conway knot in  Figure~\ref{fig3} are semi-alternating.  The main idea  in their construction is the existence of `Hamiltonian paths'. They showed that such a construction gives a minimal arc presentation if the tangles are strongly alternating. However, it is not yet known for non semi-alternating knots.

\section{Main Results}
A \emph{Montesinos link} $M \big( e; \frac{{\beta_1}}{{\alpha_1}}, \cdots , \frac{{\beta_n}}{{\alpha_n}} \big)$ has a projection as shown in  Figure~\ref{fig4} in which $e$ and $R_i$ denote the number of half twists and  a rational 2-string tangle with slope $\frac{{\beta_i}}{{\alpha_i}}$, respectively, for $i=1,\ldots,n$. If $e=0$, 
we use $M \big( \frac{{\beta_1}}{{\alpha_1}}, \cdots , \frac{{\beta_n}}{{\alpha_n}} \big)$, omitting~$e$.

\begin{figure}[!h]
  \centering
		\includegraphics[height=2.5cm]{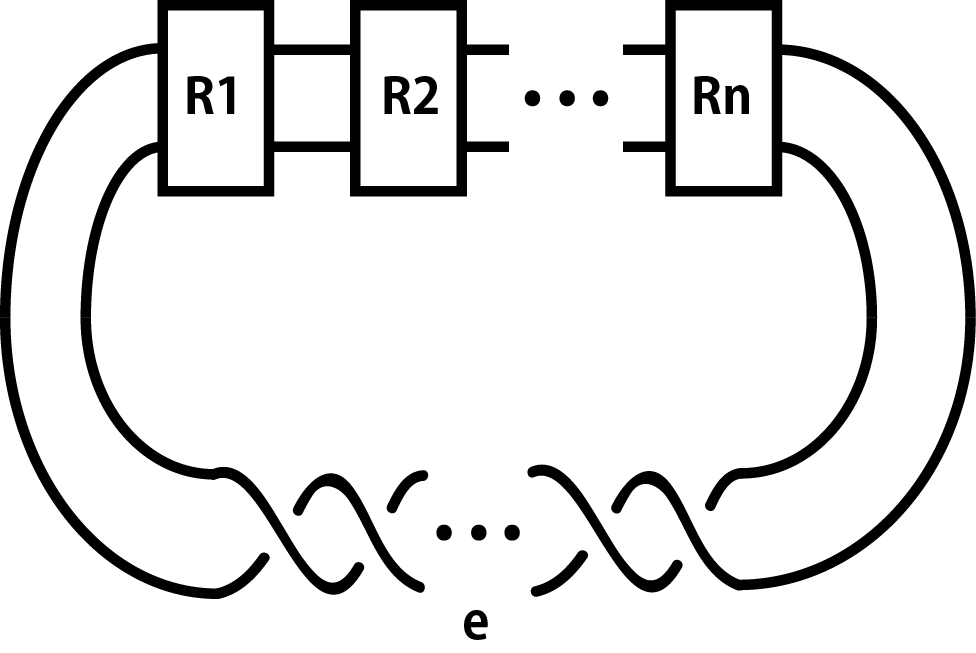}
  \caption{Montesinos link diagram}\label{fig4}
\end{figure}

We only work on Montesinos knots with four tangles, i.e., $n=4$, with $e=0$.
We focus on non semi-alternating mutant knots which have 11 or 12 crossings. We found a place to attach full twist and proved that there are infinitely many non semi-alternating mutant knot pairs and triples which have same arc index. For these mutant knots, the arc index is not changed by mutation. 

In all figures below, each square indicates a 2-string tangle of right handed twists with crossings as many as the number in it.

\begin{theorem}\label{thm3_1}
The  mutant pair $M(\frac{2}{3}, -\frac{2}{3}, \frac{2}{3}, \frac{1}{2n+2})$ and $M(-\frac{2}{3}, \frac{2}{3}, \frac{2}{3}, \frac{1}{2n+2})$ are distinct knots with  the minimal crossing number and the arc index  equal to $2n+11$, for $n\ge0$.
\end{theorem}
\begin{figure}[!ht]
  \centering
		\includegraphics[height=2.8cm]{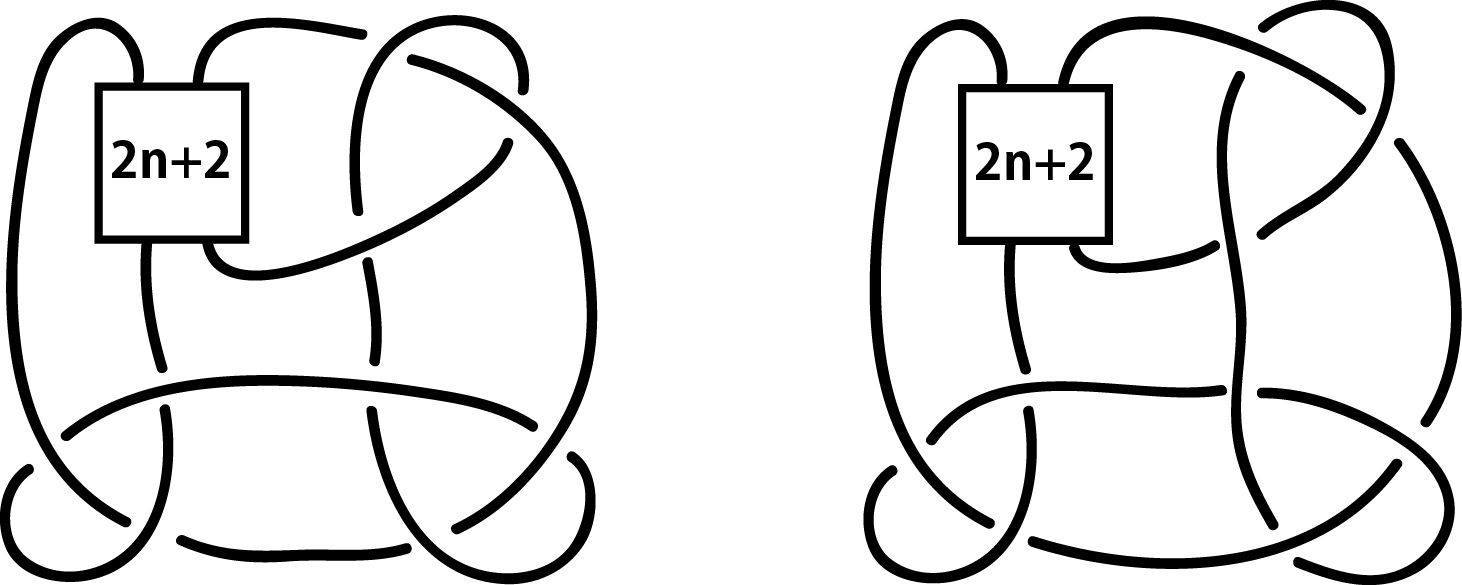}
  \caption{Mutant pair $M(\frac{2}{3}, -\frac{2}{3}, \frac{2}{3}, \frac1{2n+2})$ and $M(-\frac{2}{3}, \frac{2}{3}, \frac{2}{3}, \frac1{2n+2})$}
\end{figure}

\begin{theorem}\label{thm3_2}
The  mutant pair  $M(\frac{1}{2}, \frac{2}{3}, -\frac{2}{3}, \frac1{2n+3})$ and $M(\frac{1}{2}, -\frac{2}{3}, \frac{2}{3}, \frac1{2n+3})$ are distinct knots with  the minimal crossing number and the arc index  equal to  $2n+11$, for $n\ge0$.
\end{theorem}
\begin{figure}[!ht]
  \centering
		\includegraphics[height=2.8cm]{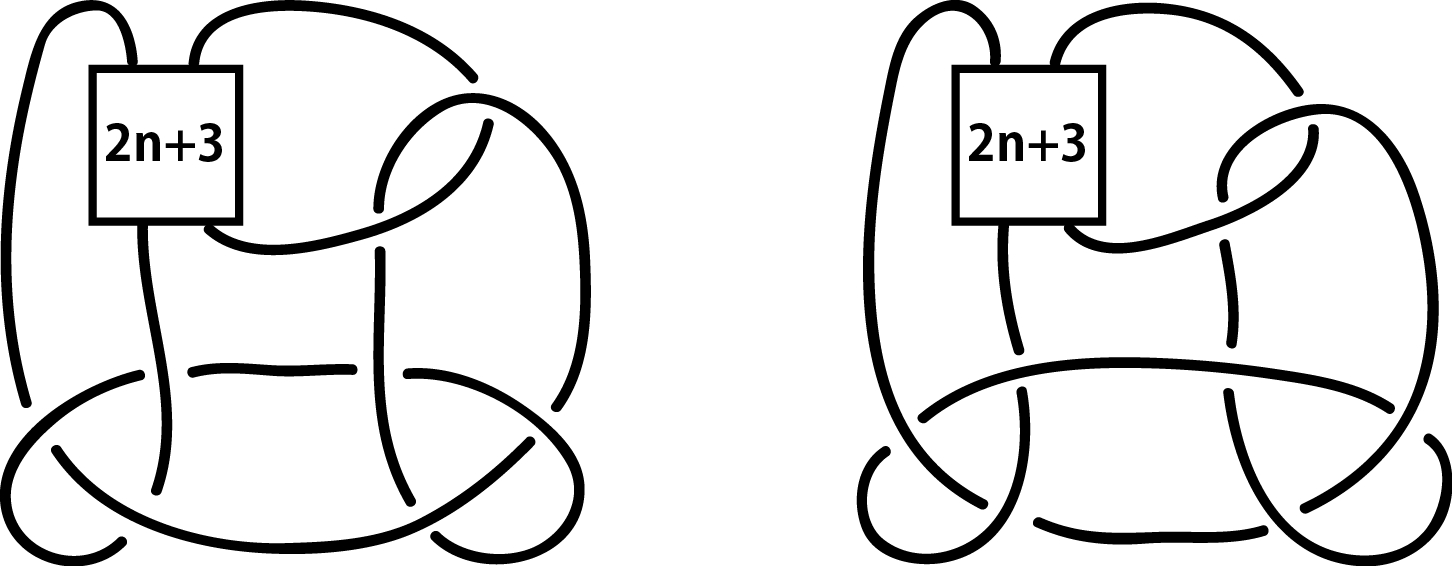}
  \caption{Mutant pair $M(\frac{1}{2}, \frac{2}{3}, -\frac{2}{3}, \frac1{2n+3})$ and $M(\frac{1}{2}, -\frac{2}{3}, \frac{2}{3}, \frac1{2n+3})$}
\end{figure}

\begin{theorem}\label{thm3_3}
The  mutant pair  $M(-\frac{2}{3}, \frac{2}{3}, \frac{2}{3}, \frac1{2n+3})$ and $M(\frac{2}{3}, -\frac{2}{3}, \frac{2}{3}, \frac1{2n+3})$ are distinct knots with  the minimal crossing number and the arc index  equal to $2n+12$, for $n\ge0$.
\end{theorem}

\begin{figure}[!ht]
  \centering
		\includegraphics[height=2.8cm]{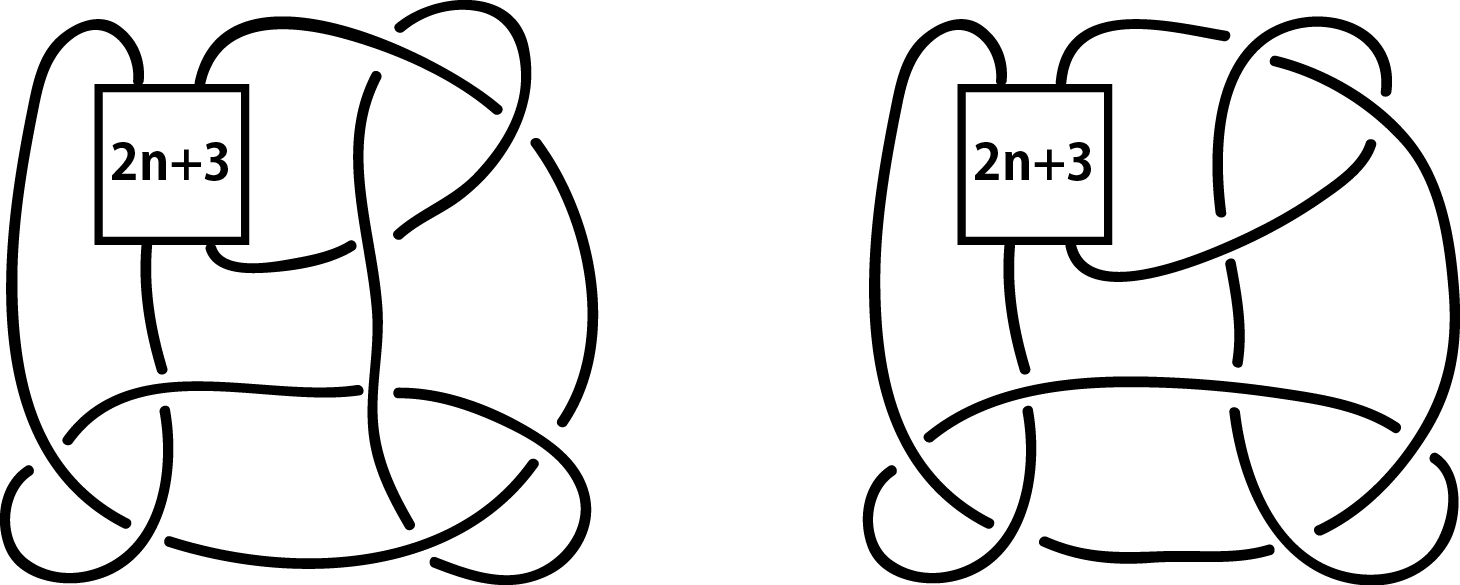}
  \caption{Mutant pair $M(-\frac{2}{3}, \frac{2}{3}, \frac{2}{3}, \frac1{2n+3})$ and $M(\frac{2}{3}, -\frac{2}{3}, \frac{2}{3}, \frac1{2n+3})$}
\end{figure}

\clearpage
\begin{theorem}\label{thm3_4}
The  mutant pair  $M(-\frac{3}{5}, \frac{2}{3}, \frac{2}{3}, \frac1{2n+2})$ and $M(\frac{2}{3}, -\frac{3}{5}, \frac{2}{3}, \frac1{2n+2})$ are distinct knots with  the minimal crossing number and the arc index  equal to $2n+12$, for $n\ge0$.
\end{theorem}

\begin{figure}[h!]
  \centering
		\includegraphics[height=2.8cm]{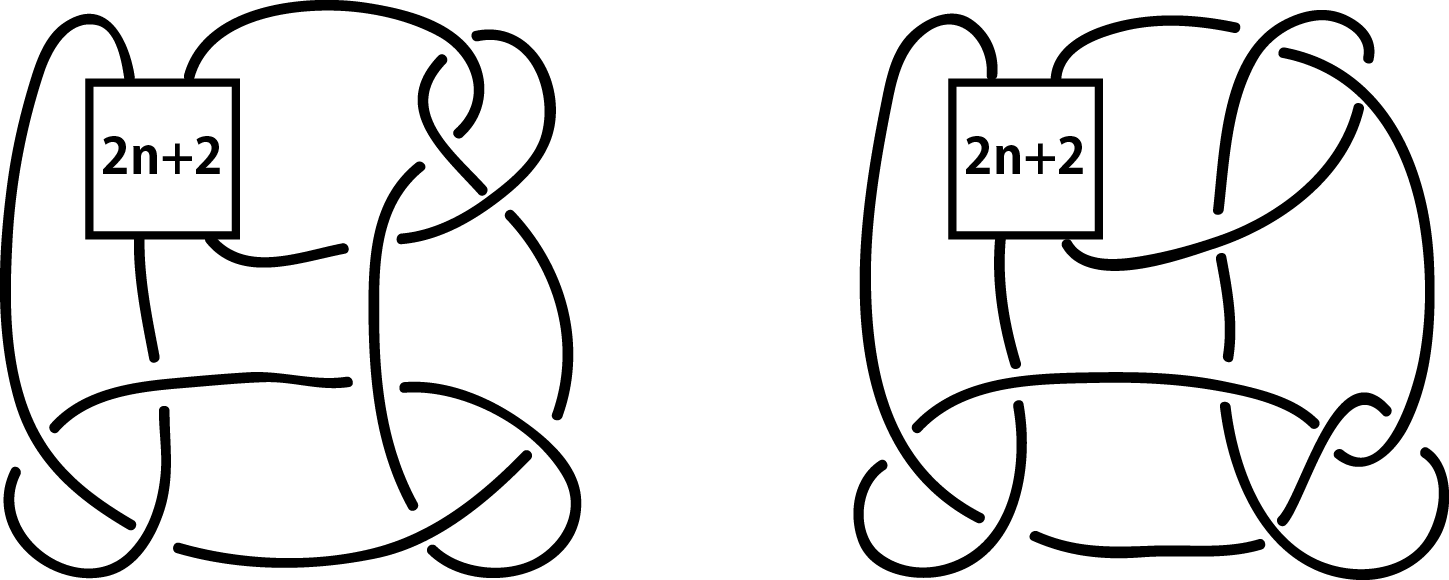}
  \caption{Mutant pair $M(-\frac{3}{5}, \frac{2}{3}, \frac{2}{3}, \frac1{2n+2})$ and $M(\frac{2}{3}, -\frac{3}{5}, \frac{2}{3}, \frac1{2n+2})$}
\end{figure}

\begin{theorem}\label{thm3_5}
The  mutant pair  $M(-\frac{3}{5}, -\frac{2}{3}, -\frac{2}{3}, \frac1{2n+2})$ and $M(-\frac{2}{3}, -\frac{3}{5}, -\frac{2}{3}, \frac1{2n+2})$ are distinct knots with  the minimal crossing number and the arc index  equal to $2n+12$, for $n\ge0$.
\end{theorem}

\begin{figure}[!ht]
  \centering
		\includegraphics[height=2.8cm]{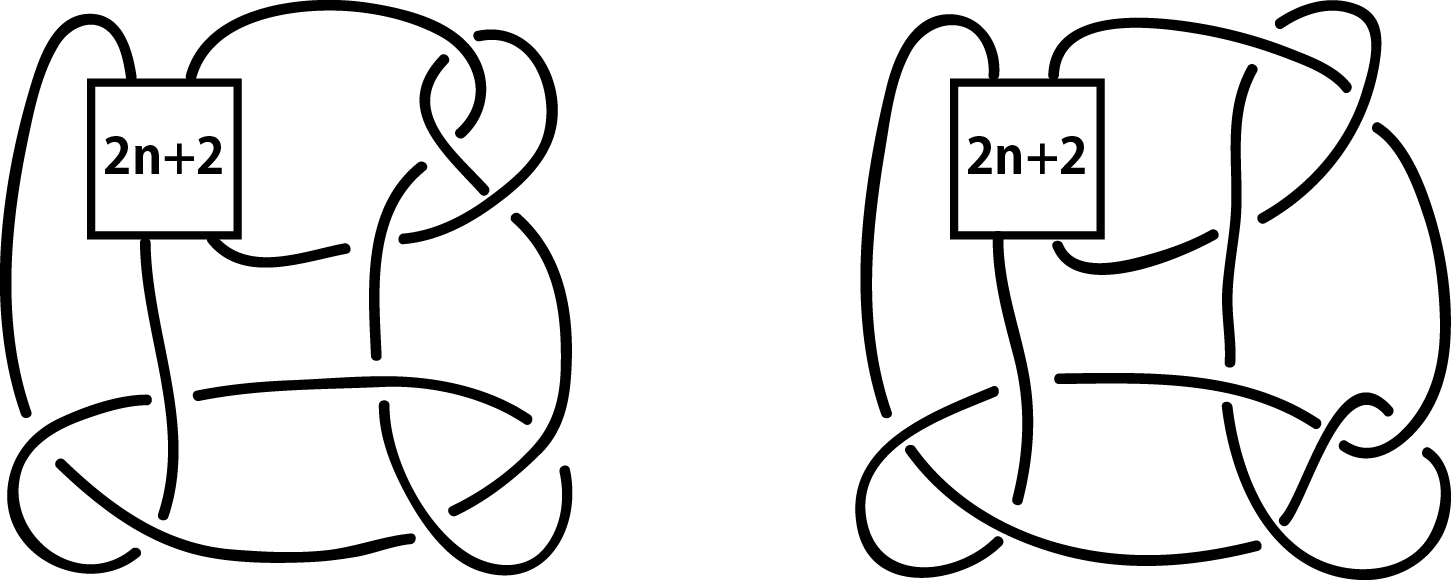}
  \caption{Mutant pair $M(-\frac{3}{5}, -\frac{2}{3}, -\frac{2}{3}, \frac1{2n+2})$ and $M(-\frac{2}{3}, -\frac{3}{5}, -\frac{2}{3}, \frac1{2n+2})$}
\end{figure}

\begin{theorem}\label{thm3_6}
The  mutant pair  
$M(\frac{1}{2}, \frac{3}{5}, -\frac{2}{3}, \frac1{2n+3})$ and 
$M(\frac{1}{2}, -\frac{2}{3}, \frac{3}{5}, \frac1{2n+3})$ are distinct knots with  the minimal crossing number and the arc index  equal to $2n+12$, for $n\ge0$.
\end{theorem}
\begin{figure}[!ht]
  \centering
		\includegraphics[height=2.8cm]{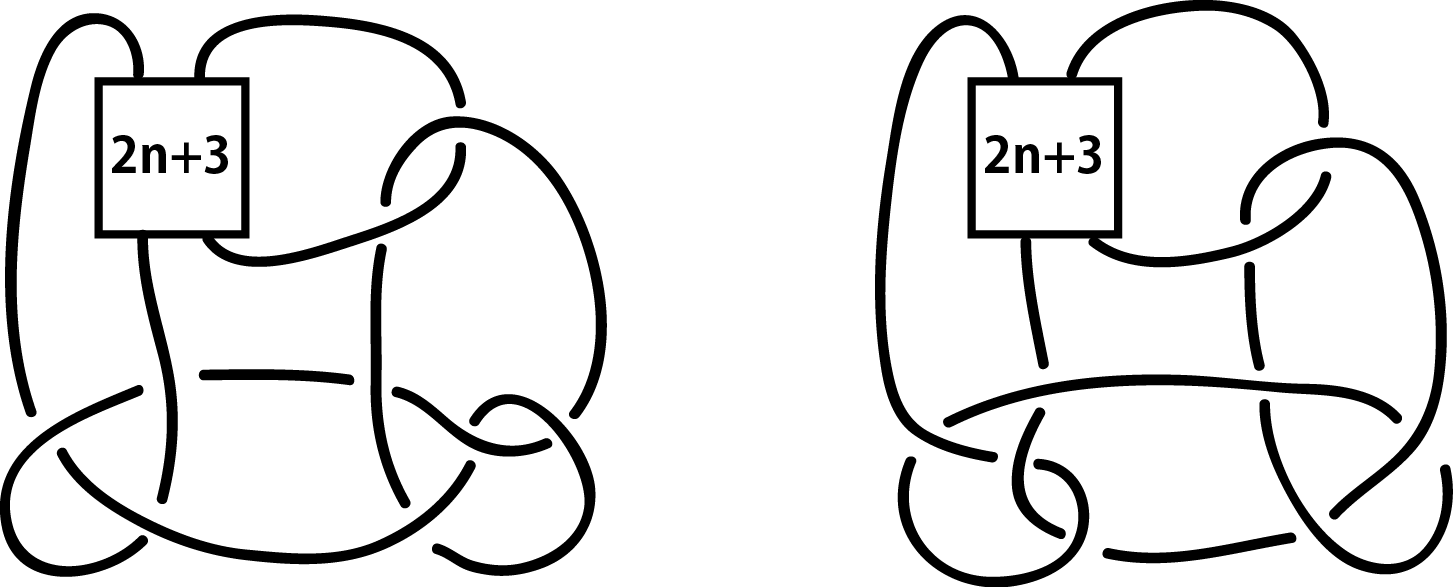}
  \caption{Mutant pair $M(\frac{1}{2}, \frac{3}{5}, -\frac{2}{3}, \frac1{2n+3})$ and $M(\frac{1}{2}, -\frac{2}{3}, \frac{3}{5}, \frac1{2n+3})$}
\end{figure}

\clearpage
\begin{theorem}\label{thm3_7}
The  mutant triple  $M(\frac{3}{5}, -\frac{2}{3}, \frac{2}{3}, \frac1{2n+2})$, $M(\frac{3}{5}, \frac{2}{3}, -\frac{2}{3}, \frac1{2n+2})$, and \break
$M(\frac{2}{3}, \frac{3}{5}, -\frac{2}{3}, \frac1{2n+2})$ are distinct knots with  the minimal crossing number and the arc index  equal to $2n+12$, for $n\ge0$.
\end{theorem}
\begin{figure}[!ht]
  \centering
		\includegraphics[height=2.8cm]{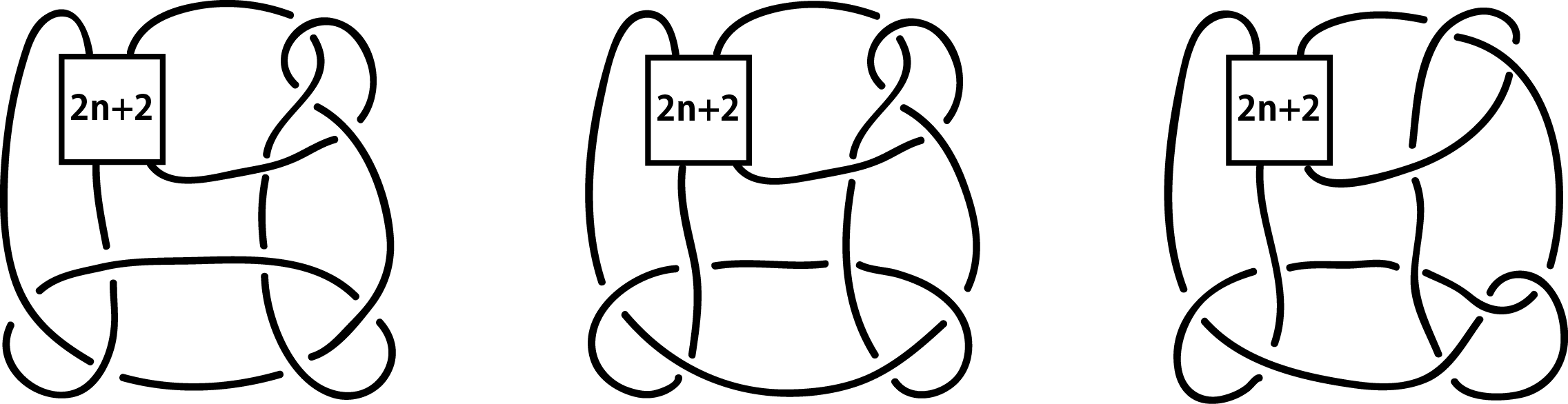}
  \caption{Mutant triple $M(\frac{3}{5}, -\frac{2}{3}, \frac{2}{3}, \frac1{2n+2})$, $M(\frac{3}{5}, \frac{2}{3}, -\frac{2}{3}, \frac1{2n+2})$, and $M(\frac{2}{3}, \frac{3}{5}, -\frac{2}{3}, \frac1{2n+2})$}
\end{figure}

\section{Proof of  Theorems~\ref{thm3_1}--\ref{thm3_7}}
By the theorem below, we easily see that the mutants in each of Theorems~\ref{thm3_1}--\ref{thm3_7} are distinct knots.

\begin{theorem}[\cite{9} Classification of Montesinos links]\label{thm3_9}
Montesinos links with $r$ rational tangles, $r \ge 3$ and $\sum\limits_{j=1}^r {\frac{1}{{\alpha}_{j}}} \le r-2$, are classified by the ordered set of fractions  
$\big( \frac{{\beta_1}}{{\alpha_1}} \mod1, \cdots ,$ 
$ \frac{{\beta_r}}{{\alpha_r}} \mod1 \big)$, up to cyclic permutations and reversal of order, together with the rational number $e_0 = e + \sum\limits_{j=1}^r {\frac{{\beta}_{j}}{{\alpha}_{j}}}$.
\end{theorem}

Now we show that each knot in Theorems~\ref{thm3_1}--\ref{thm3_7} has an arc presentation with as many arcs as the mentioned arc index. In Figure~\ref{fig:2n+11 arcs} and Figure~\ref{fig:2n+12 arcs}, the thick curves are binding circles which have the pattern shown in Figure~\ref{fig5} near each box. One can easily count  that the number of arcs is as mentioned for each case.

\begin{figure}[h!]
  \centering
		\includegraphics[height=2.5cm]{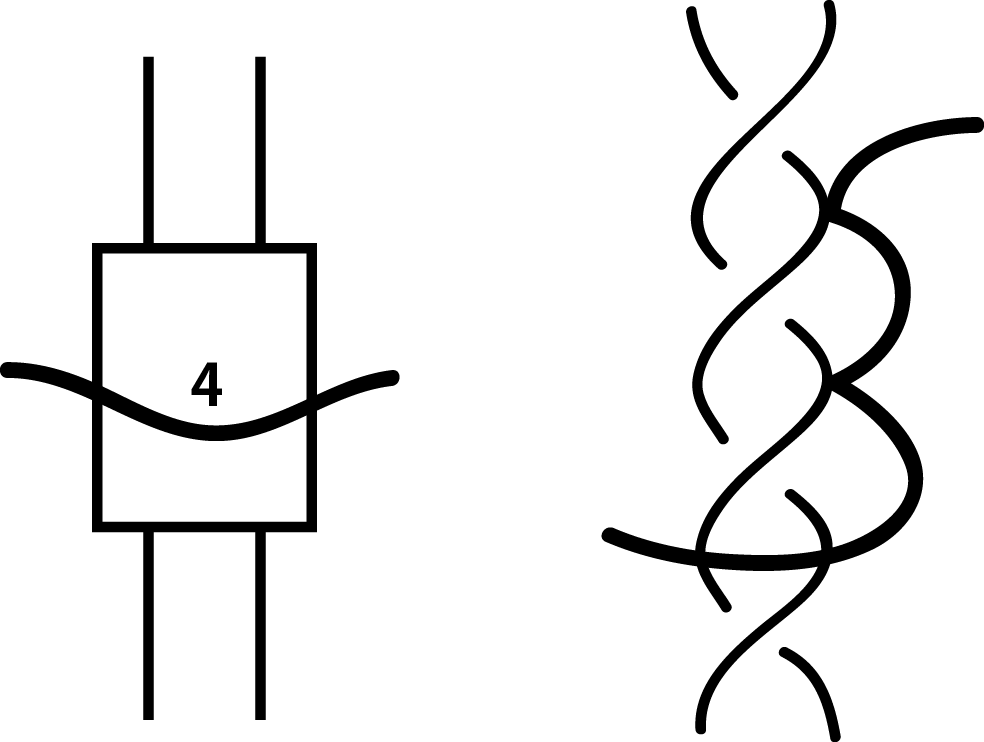}
  \caption{Deforming the curve $C$ near a twist box}\label{fig5}
\end{figure}

\begin{figure}[h!]\footnotesize
  \centerline{
		\includegraphics[height=2cm]{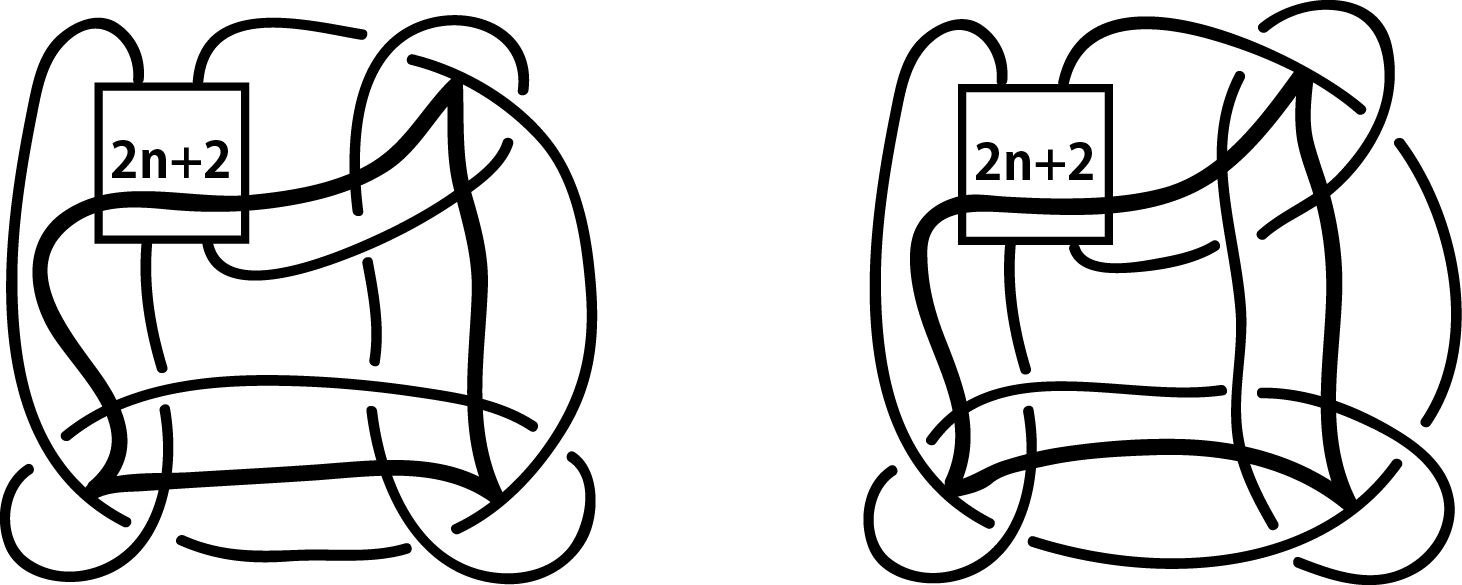}
		\hspace{10mm}
		\includegraphics[height=2cm]{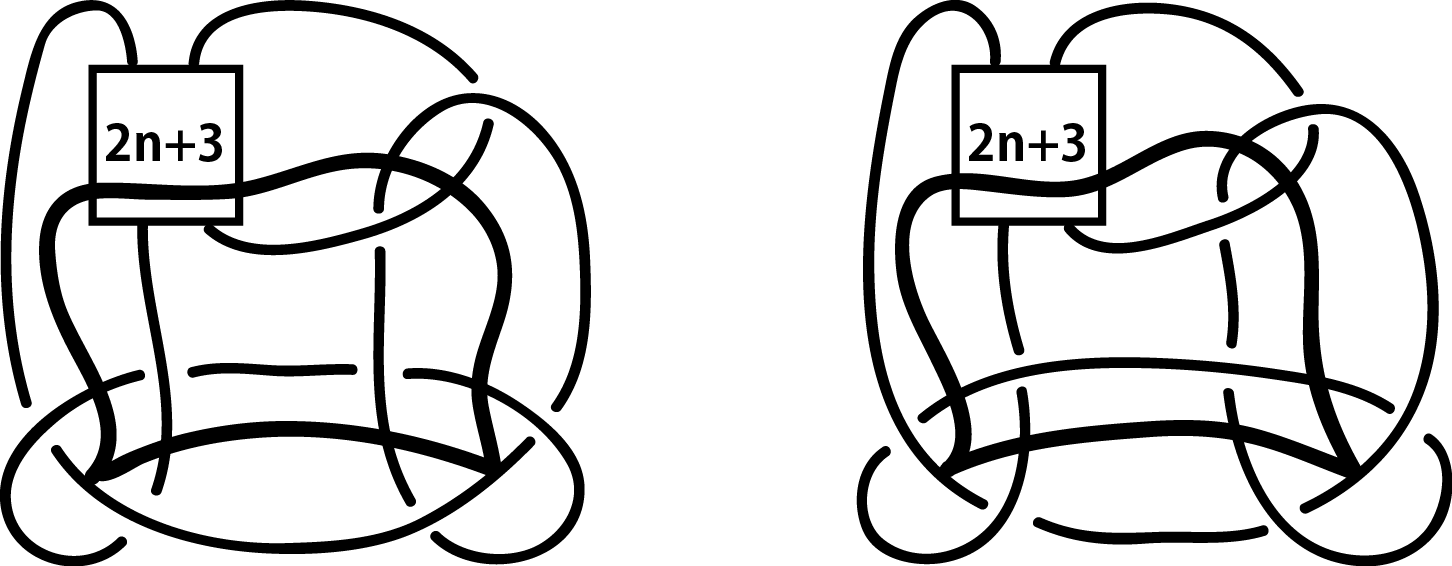}}

		\centerline{Theorem~\ref{thm3_1}\hspace{45mm}Theorem~\ref{thm3_2}}
  \caption{Arc presentations with $2n+11$ arcs}\label{fig:2n+11 arcs}
\end{figure}

\begin{figure}[h!]\footnotesize
  \centering
		\includegraphics[height=2cm]{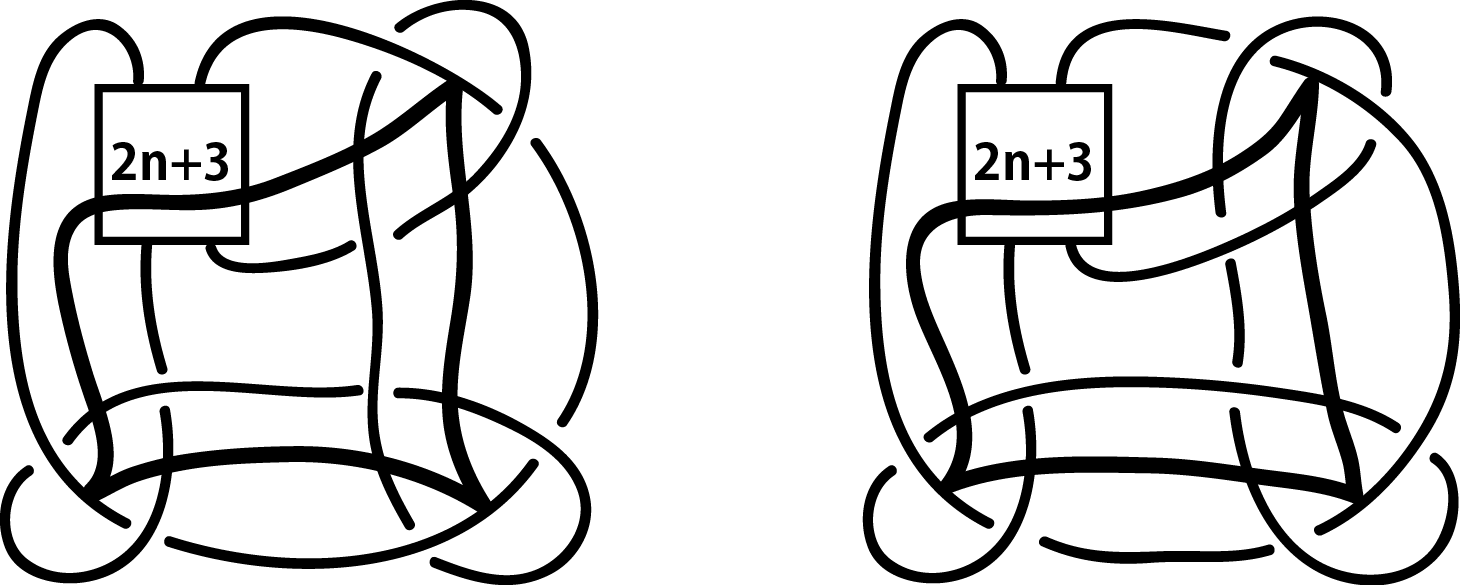}
\hspace{10mm}
		\includegraphics[height=2cm]{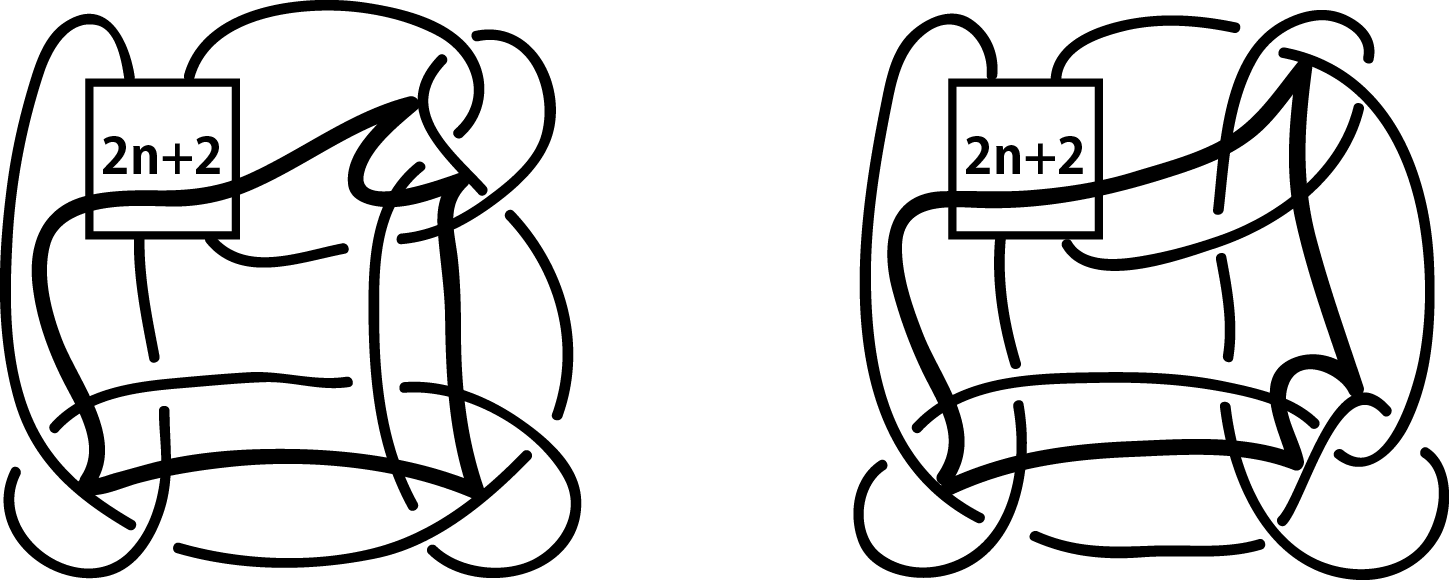}
		\hbox{Theorem~\ref{thm3_3}\hspace{45mm}Theorem~\ref{thm3_4}}
\\
\vspace{3mm}
		\includegraphics[height=2cm]{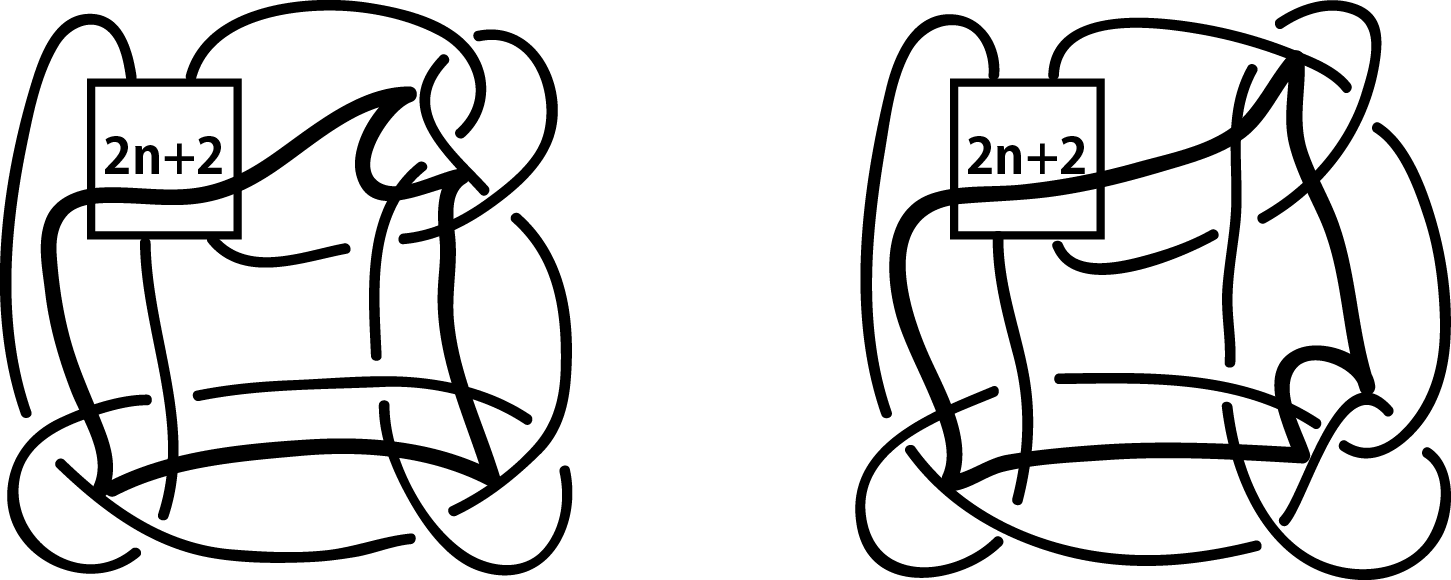}
\hspace{10mm}
		\includegraphics[height=2cm]{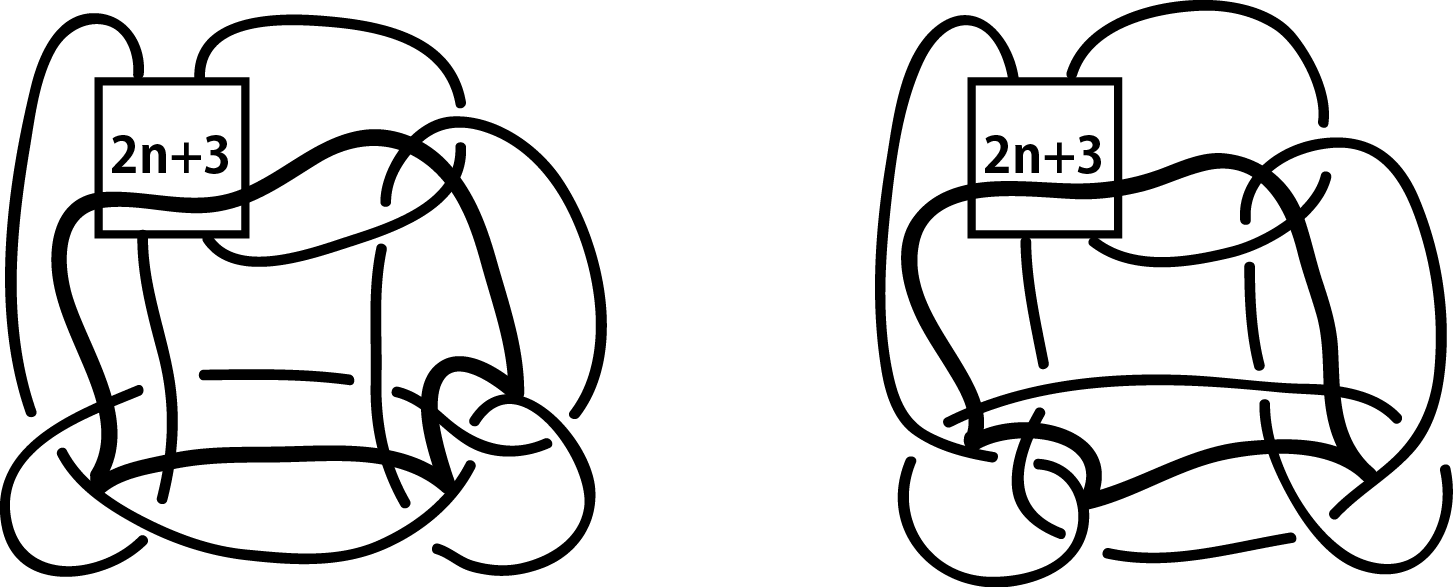}
		\hbox{Theorem~\ref{thm3_5}\hspace{45mm}Theorem~\ref{thm3_6}}
\\
\vspace{3mm}
		\includegraphics[height=2cm]{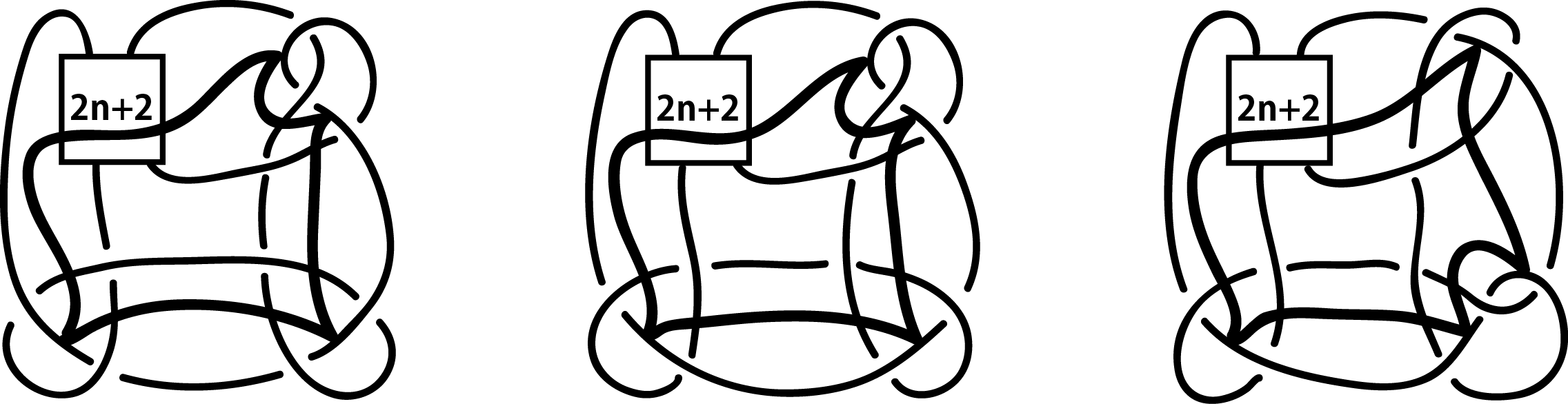}
\hbox to 0.7\textwidth{\hfill Theorem~\ref{thm3_7}\hfill}
  \caption{Arc presentations with $2n+12$ arcs}\label{fig:2n+12 arcs}
\end{figure}

The numbers of arcs we just counted are upper bounds of the arc index. Now it remains to show that the same numbers are lower bounds.
By Proposition~\ref{prop:breadth+2}, we only need to show that $\breadth_a(F_L(a,z))$ is 2 less than the number of arcs we counted. 

Table~1  shows the case $n=0$ for Theorems~\ref{thm3_1}--\ref{thm3_7}, where the names are as in the KnotInfo site~\cite{knotinfo}.

\begin{table}[ht]
{\begin{tabular}{@{}cllc@{}} \hline
\noalign{\smallskip}Theorem  &Montesinos knots&Name&Arc Index \\
\hline
\noalign{\smallskip}
\ref{thm3_1}&$M(\frac{2}{3}, -\frac{2}{3}, \frac{2}{3}, \frac{1}{2})$&$11n71^*$&11\vrule width0pt depth6pt\\
\noalign{\smallskip}
&$M(-\frac{2}{3}, \frac{2}{3}, \frac{2}{3}, \frac{1}{2})$ &$11n75$&11 \\
\noalign{\smallskip}\hline
\noalign{\smallskip}
\ref{thm3_2}&$M(\frac{1}{2}, \frac{2}{3}, -\frac{2}{3}, \frac1{3})$ &$11n76^*$&11\vrule width0pt depth6pt\\
\noalign{\smallskip}
&$M(\frac{1}{2}, -\frac{2}{3}, \frac{2}{3}, \frac1{3})$ &$11n78$&11 \\
\noalign{\smallskip}\hline
\noalign{\smallskip}
\ref{thm3_3}&$M(-\frac{2}{3}, \frac{2}{3}, \frac{2}{3}, \frac1{3})$ &$12n553$&12\vrule width0pt depth6pt\\
\noalign{\smallskip}
&$M(\frac{2}{3}, -\frac{2}{3}, \frac{2}{3}, \frac1{3})$ &$12n556^*$&12 \\
\noalign{\smallskip}\hline
\noalign{\smallskip}
\ref{thm3_4}&$M(-\frac{3}{5}, \frac{2}{3}, \frac{2}{3}, \frac1{2})$ &$12n55^*$&12\vrule width0pt depth6pt\\
\noalign{\smallskip}
&$M(\frac{2}{3}, -\frac{3}{5}, \frac{2}{3}, \frac1{2})$ &$12n223^*$&12 \\
\noalign{\smallskip}\hline
\noalign{\smallskip}
\ref{thm3_5}&$M(-\frac{3}{5}, -\frac{2}{3}, -\frac{2}{3}, \frac1{2})$ &$12n58^*$&12\vrule width0pt depth6pt\\
\noalign{\smallskip}
&$M(-\frac{2}{3}, -\frac{3}{5}, -\frac{2}{3}, \frac1{2})$ &$12n222$&12 \\
\noalign{\smallskip}\hline
\noalign{\smallskip}
\ref{thm3_6}&$M(\frac{1}{2}, \frac{3}{5}, -\frac{2}{3}, \frac1{3})$ &$12n64$&12\vrule width0pt depth6pt\\
\noalign{\smallskip}
&$M(\frac{1}{2}, -\frac{2}{3}, \frac{3}{5}, \frac1{3})$ &$12n261$&12 \\
\noalign{\smallskip}\hline
\noalign{\smallskip}
\ref{thm3_7}&$M(\frac{3}{5}, -\frac{2}{3}, \frac{2}{3}, \frac1{2})$ &$12n60$&12\vrule width0pt depth6pt\\
\noalign{\smallskip}
&$M(\frac{3}{5}, \frac{2}{3}, -\frac{2}{3}, \frac1{2})$ &$12n61$&12\vrule width0pt depth6pt\\
\noalign{\smallskip}
&$M(\frac{2}{3}, \frac{3}{5}, -\frac{2}{3}, \frac1{2})$ &$12n219$&12\vrule width0pt depth6pt\\
\hline
\end{tabular}}
\medskip
\caption{The case $n=0$. ($^*$ means mirror image)}
\end{table}

\noindent
Now we go by induction on $n\ge1$.
The lemma below handles the case of Theorem~\ref{thm3_1}.
\begin{lemma} \label{lemma3_11}
 ${\breadth}_a (F_{M(\frac{2}{3}, -\frac{2}{3}, \frac{2}{3}, \frac{1}{2n+2})}) ={\breadth}_a (F_{M(-\frac{2}{3}, \frac{2}{3}, \frac{2}{3}, \frac{1}{2n+2})}) = 2n+9$ for $n\ge1$.
\end{lemma}

\begin{proof}

Recall that the Kauffman polynomial of a link $L$ is defined by $F_L(a,z) = a^{- w(D)} $
$\Lambda_{D}(a,z)$. The polynomial ${\Lambda}_D$ is of the form 
\begin{center}
$\Lambda_{D}(a,z) = \sum\limits_{i=m}^n {f_i(z) a^i}$
\end{center}
where $m,n$ are integers with $m \le n$ and $f_i(z)$'s are polynomials in $z$ with integer coefficients such that $f_m(z) \neq 0$ and $f_n(z) \neq 0$.
We use the notation
\begin{center}
${\sum\limits_{i=m}^n {f_i(z) a^i} = [c_1 z^{h_m}a^m, c_2 z^{h_n}a^n]}$
\end{center}
where $z^{h_m}$, $z^{h_n}$ are the lowest degree terms in $f_m(z)$ and $f_n(z)$, respectively.
$c_1$, $c_2$ are non zero integer coefficients.

For example, we write
\begin{center}
$(z -3z^3 + z^5)a^{-1} + (-z^{-1} + 3z^3 )a^{2} + (4z^2)a^{3} = [z a^{-1} , 4z^2 a^3]$
\end{center}

Since the Kauffman polynomial is invariant under mutations, we only compute one of the two~\cite{3a}. For simplicity of exposition, we write $\Lambda_{2n}$ as the $\Lambda$-polynomial of $M(\frac{2}{3}, -\frac{2}{3}, \frac{2}{3}, \frac{1}{2n+2})$. Since ${\breadth}_a (F_{L})={\breadth}_a(\Lambda_{D}(a,z))$, it is enough to show that $\Lambda_{2n} = [z^{3} a^{-(2n+4)}, (-1)^n 2 z^2 a^{5}] $ for $n\ge1$. In the case of $n=1$, we obtain $\Lambda_{2} (a,z)$ as $[z^3 a^{-6}, -2 z^2 a^{5}]$. Here, we used the  KNOT program of Kodama~\cite{Kodama}. One can see the whole polynomial in the appendix.  Now we assume that the condition holds for $n-1$. For the uppermost crossing inside the twist box, consider the skein relation. Figure~\ref{fig8} shows $D_{+}$, $D_{-}$, $D_{=}$, $D_{\|}$ respectively.

\begin{figure}[th]
  \centering
		\includegraphics[height=3cm]{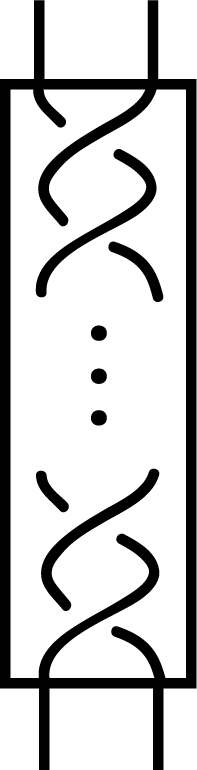}
\hspace{5mm}
		\includegraphics[height=3cm]{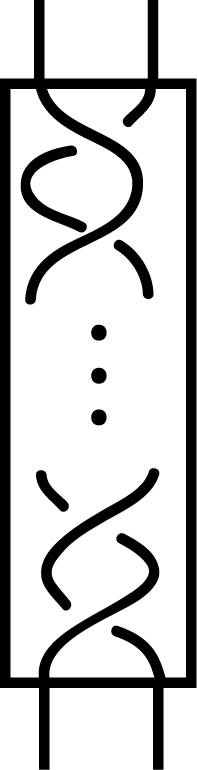}
\hspace{5mm}
		\includegraphics[height=3cm]{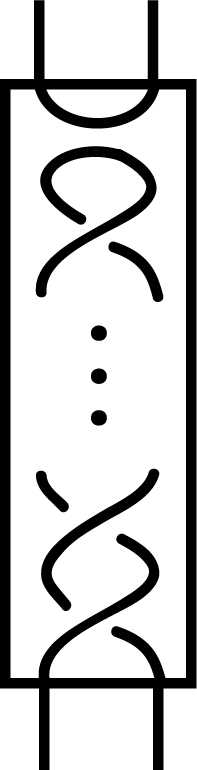}
\hspace{5mm}
		\includegraphics[height=3cm]{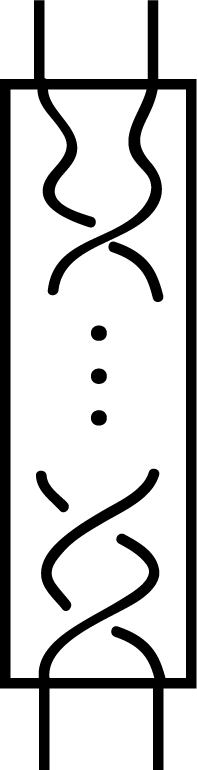}
  \caption{Skein relation in twist box}\label{fig8}
\end{figure}

If we write $\Lambda_{D_{+}}$ as $\Lambda_{2n}$ then $\Lambda_{D_{-}}$ is equal to $\Lambda_{2n-2}$
since $\Lambda (a,z)$ does not change by the second Reidemeister move. The diagram $D_{=}$ has a curl with $2n+1$ right-handed twist. After $2n+1$ times of the first Reidemeister moves, we obtain  $\Lambda_{D'}$ from $\Lambda_{D_=}$ where diagram $D'$ is shown in the Figure \ref{fig9}. For $\Lambda_{D_{\|}}$, we denote by $\Lambda_{D_{2n-1}}$. By the second property of $\Lambda(a,z)$, we have 
\begin{center}
${\Lambda}_{2n} = z {\Lambda}_{D_{2n-1}} + z a^{-(2n+1)} {\Lambda}_{D'} - {\Lambda}_{2n-2}$. 
\end{center}

From this recurrence relation, we can obtain the following equation where ${\Lambda}_{D_1}$ is shown in the Figure~\ref{fig9}.
\begin{align}
 {\Lambda}_{2n} &=  (-1)^{n+1} z {\Lambda}_{D_1} + (a^{-2n} - a^{-(2n-2)}+ \cdots + (-1)^na^{-4}) z^{2} {\Lambda}_{D'} \nonumber  \\
 &\qquad {} + ({\Lambda}_{2n-2} - {\Lambda}_{2n-4} + \cdots + (-1)^{n} {\Lambda}_{2}) z^{2} \nonumber \\
 &\qquad {} + za^{-(2n+1)}{\Lambda}_{D'} - {\Lambda}_{2n-2} \nonumber 
\end{align}

\begin{figure}[ht]
  \centering
		\includegraphics[height=2cm]{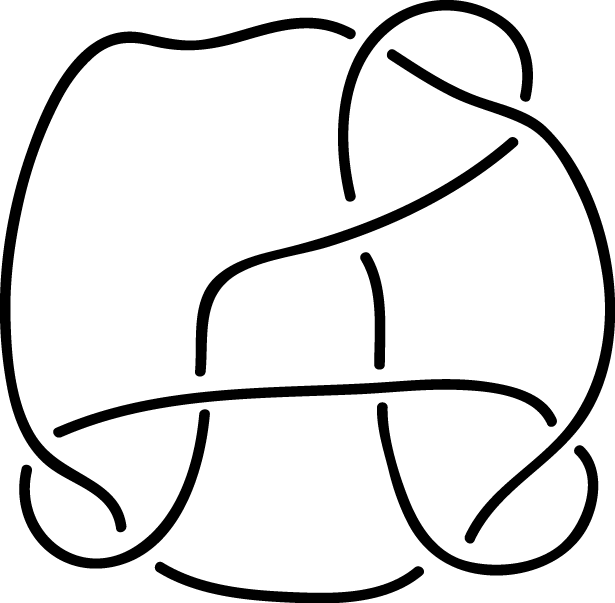} \hspace{10mm}
		\includegraphics[height=2cm]{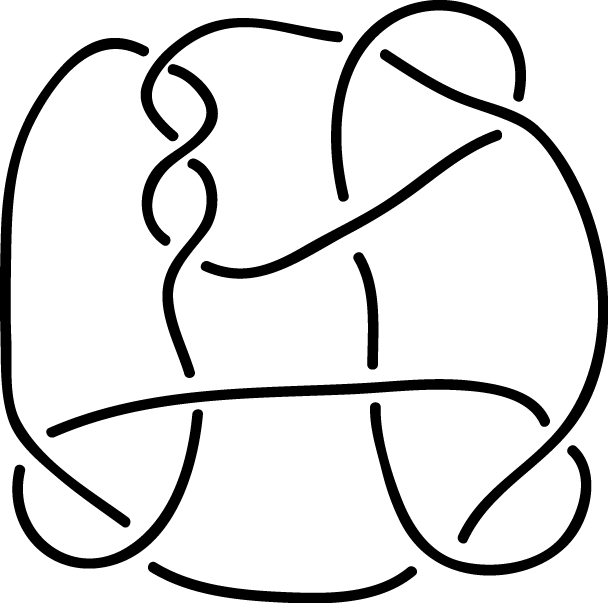}
  \caption{Diagrams $D'$ and $D_1$}\label{fig9}
\end{figure}

We found the highest and lowest degrees of $a$ in the five terms of the right hand side. ${\Lambda}_{D_1}$ and ${\Lambda}_{D'}$ can be calculated as ${\Lambda}_{D_1} = [z^3 a^{-5}, 4 z^3 a^{5}]$ and ${\Lambda}_{D'} = [z^2 a^{-3}, z a^{4}]$. Also by the assumption, ${\Lambda}_{2n-2} = [z^3 a^{-(2n+2)}, (-1)^{n-1}2 z^2 a^{5}]$. Therefore, we have

\begin{align*}
    (-1)^{n+1} z {\Lambda}_{D_1} &= [(-1)^{n+1} z^4 a^{-5}, (-1)^{n+1} 4 z^4 a^{5}]\\
  (a^{-2n} - a^{-(2n-2)}+ \cdots + (-1)^na^{-4}) z^{2} {\Lambda}_{D'} &= [z^4 a^{-(2n+3)}, z^3]\\
  ({\Lambda}_{2n-2} - {\Lambda}_{2n-4} + \cdots + (-1)^{n} {\Lambda}_{2}) z^{2} &= [z^5 a^{-(2n+2)}, (-1)^{n-1}2 (n-1) z^4 a^{5}]\\
  za^{-(2n+1)}{\Lambda}_{D'} &= [z^3 a^{-(2n+4)}, z^2 a^{-(2n-3)}] \\
  -{\Lambda}_{2n-2} &= [- z^3 a^{-(2n+2)}, (-1)^{n}2 z^2 a^{5}]
\end{align*}

The lowest term of $za^{-(2n+1)}{\Lambda}_{D'}$ and the highest term of $-{\Lambda}_{2n-2}$ together determine $\Lambda_{2n}$ as $\Lambda_{2n} = [z^{3} a^{-(2n+4)}, (-1)^n 2 z^2 a^{5}]$. This completes the proof.
\end{proof}

Statements similar to Lemma~\ref{lemma3_11} can be shown for Theorems~\ref{thm3_2}--\ref{thm3_7}, using the polynomials in the appendix.

\section{Knots in Theorems~\ref{thm3_1}--\ref{thm3_7} are non semi-alternating.}
Finally, we discuss that the knots in Theorems~\ref{thm3_1}--\ref{thm3_7} are non semi-alternating.
Jones polynomial is determined by the properties that it takes the value $1$ on a diagram of the unknot and satisfies the following skein relation:
$$t V_{L_{+}} - t^{-1}V_{L_{-}} = (t^{- \frac{1}{2}} - t^{\frac{1}{2}}) V_{L_{0}}$$
where $L_{+}$, $L_{-}$ and $L_{0}$ are three oriented link diagrams that are identical except in one small region where they differ by the crossing changes or smoothing. Lickorish and Thislethwaite found a method that can determine whether the given knot is non semi-alternating. 

\begin{proposition}[\cite{10}]\label{prop3_12}
Let $L$ be a link admitting semi-alternating diagram $D$ with $n$ crossing. Then the Jones polynomial $V_{L}(t)$ is non-alternating, its extreme coefficitents are $\pm 1$, and its breadth is $n-1$.
\end{proposition}

$V_{L}(t)$ is of the form
\begin{center}
$V_{L}(t) = \sum\limits_{i=m}^n {c_i t^i}$
\end{center}
Similar to the Kauffman polynomial, we use the notation
$$V_{L}(t) = [c_m t^m, c_n t^n]$$
where $c_m$, $c_n$ are non zero coefficients. The difference $n-m$ is called the \emph{breadth} of $V_L$. For the uppermost crossing inside the twist box, there are two types of skein relation as shown in the Figure~\ref{fig10}.
\begin{figure}[!ht]
  \centering
		\includegraphics[height=3cm]{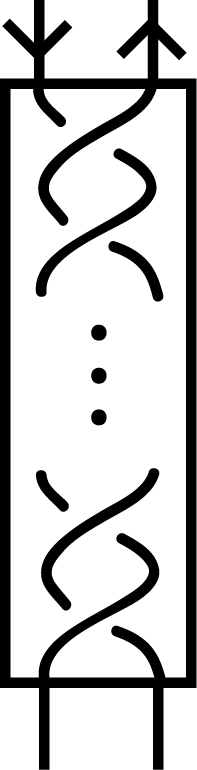}
\hspace{5mm}
		\includegraphics[height=3cm]{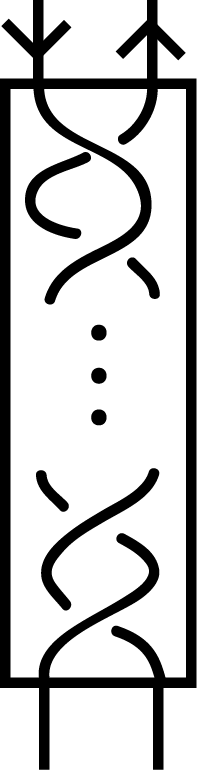}
\hspace{5mm}
		\includegraphics[height=3cm]{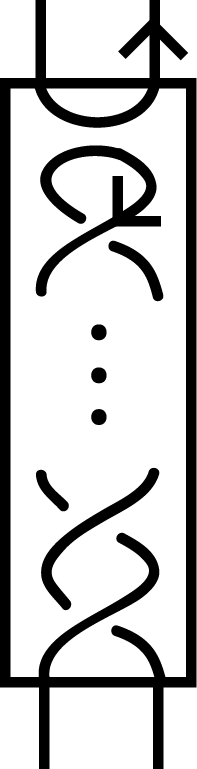}
\hspace{20mm}
		\includegraphics[height=3cm]{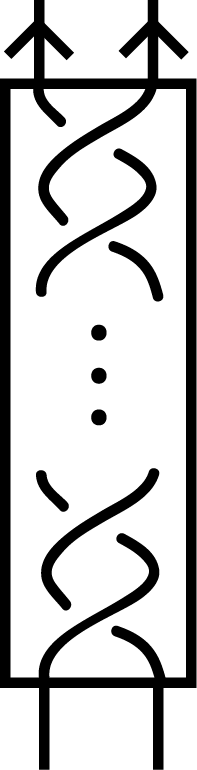}
\hspace{5mm}
		\includegraphics[height=3cm]{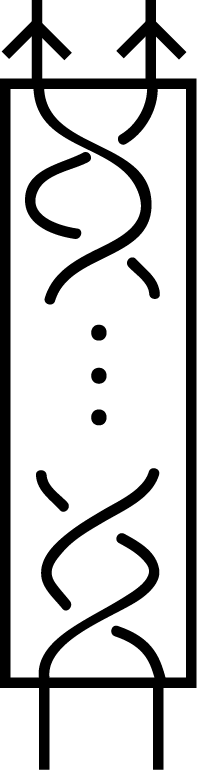}
\hspace{5mm}
		\includegraphics[height=3cm]{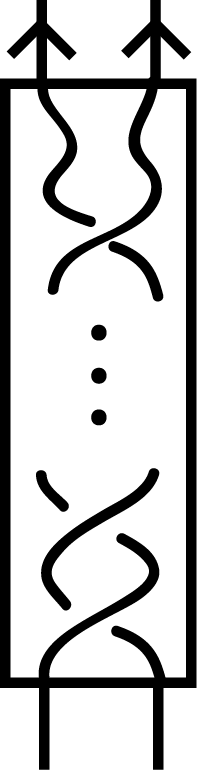}

Type1 \hspace{50mm}  Type2

  \caption{Two types of oriented skein relations in a twist box}\label{fig10}
\end{figure}

To show that the knots in  Theorem \ref{thm3_1} are 
non semi-alternating, we prove Lemma~\ref{lemma3_13}. As the Jones polynomial is invariant under mutations, we only need to handle one of the pair.

\begin{lemma}\label{lemma3_13}
The breadth of $V_{L}(t)$  is equal to 
$2n+9$ where $ L = M(\frac{2}{3}, -\frac{2}{3}, \frac{2}{3}, \frac{1}{2n+2})$.
\end{lemma}
\begin{proof}
In the case $n=0$, $L=11n71^*$ (See Table 1.) and 
$$V_L(t)=-2+5t   -7t^2   +11t^3  -10t^4   +10t^5   -9t^6    +5t^7   -3t^8    +t^9$$
having breadth 9.
For $n\ge1$, we will show that $V_{L}(t) = [- t^{0} , t^{2n+9}]$ by an induction on $n$. To use a recurrence formula, $V_{L}(t)$ is denoted by $V_{2n}(t)$. If $n=1$, we have  $V_{2}(t) = [- t^{0} , t^{11}]$. For the whole polynomial, see appendix. Now assume that $n>1$ and the condition holds for $n-1$. Since Jones polynomial is not affected by the Reidemeiter moves, we obtain the following skein relation:
$$V_{2n}(t) =  t^{2}V_{2n-2}(t) - (t^{\frac{1}{2}} - t^{ \frac{3}{2}}) V_{D'}(t)$$
where $D'$ is same as Figure \ref{fig9}, but has orientation. Here we used the type 1 skein relation.

By the assumption,  $V_{2n-2}(t) = [- t^{0} , t^{2n+7}]$ and $V_{D'}(t) = [ t^{- \frac{1}{2}}, - t^{\frac{13}{2}}]$. Thus, $t^{2}V_{2n-2}(t) = [- t^{2} , t^{2n+9}]$ and $- (t^{\frac{1}{2}} - t^{\frac{3}{2}}) V_{D'}(t) = [- t^{0} , t^{8}]$. Therefore we have $V_{2n}(t) = [- t^{0} , t^{2n+9}]$.
\end{proof}

A statement similar to Lemma~\ref{lemma3_13} can be given to show that the knots in Theorem~\ref{thm3_3} are non semi-alternating.
To show that the knots in Theorem~\ref{thm3_2} are non semi-alternating, we use the next lemma.

\begin{lemma}\label{lemma3_14}
The breadth of $V_{L}(t)$  is equal to 
$2n+9$ where $ L = M(\frac{1}{2}, \frac{2}{3}, -\frac{2}{3}, \frac1{2n+3})$.
\end{lemma}

\begin{proof}
In the case $n=0$, $L=11n76^*$ (See Table 1.) and 
$$V_L(t)=
t^{-8}   -4t^{-7}  + 5t^{-6}   -7t^{-5}  + 8t^{-4}   -6t^{-3}  + 7t^{-2}   -4t^{-1}  + 2   -t
$$
having breadth 9.
For $n\ge1$, we will show that $V_{L}(t) = [t^{-3n-8} , - t^{-n+1}]$ by an induction on $n$.
To use a recurrence formula, $V_{L}(t)$ is denoted by $V_{2n}(t)$. When $n=1$, we have $V_2 (t) = [t^{-11} , -t^0]$. For the whole polynomial, see appendix. Now assume that $n>1$ and the condition holds for $n-1$. For the uppermost crossing inside the twist box, we obtain a skein relation as follows: 
$$V_{2n}(t) =  t^{-2}V_{2n-2}(t) + (t^{- \frac{3}{2}} - t^{- \frac{1}{2}}) V_{D_{2n-1}}(t)$$
Here we used the type 2 skein relation. From this recurrence relation, we have the following equation:
$$V_{2n}(t) =  t^{-2}V_{2n-2}(t) + (t^{- \frac{3}{2}} - t^{- \frac{1}{2}})^{2} \sum\limits_{i=1}^{n-1}{t^{-2n+2i+2}V_{2i}(t)} +(t^{- \frac{3}{2}} - t^{- \frac{1}{2}}) t^{-2n+2} V_{1}(t)$$
By the assumption,  $V_{2n-2}(t) = [t^{-3n-5} , - t^{-n+2}]$ and $V_{1}(t) = [t^{- \frac{19}{2}}, t^{\frac{1}{2}}]$. So we have

\begin{align*}
 t^{-2}V_{2n-2}(t) &= [ t^{-3n-7} , - t^n ]\\
(t^{- \frac{3}{2}} - t^{- \frac{1}{2}})^{2} \sum\limits_{i=1}^{n-1}{t^{-2n+2i+2}V_{2i}(t)} &= [t^{-3n-8} , - t^{-n+1}]\\
 (t^{- \frac{3}{2}} - t^{- \frac{1}{2}})  t^{-2n+2}V_{D_1}(t) &= [t^{-2n-9} , - t^{-2n+2}].
\end{align*}
Thus $V_{2n}(t) = [t^{-3n-8} , - t^{-n+1}]$. This completes the proof.
\end{proof}

Statements similar to Lemma~\ref{lemma3_14} can be given to show that the knots in Theorems~\ref{thm3_4}--\ref{thm3_7} are non semi-alternating.

\clearpage
\section*{Appendix}

\subsection*{Polynomials to prove Theorem \ref{thm3_1}}
$$
\begin{aligned}
\Lambda_{2n} & = [z^3 a^{-2n+4}, (-1)^{n} 2 z^2 a^5]\\
\Lambda_{2} & = (z^8 - 5z^6 + 7z^4 - 2z^2)a^{5}+ (3z^9 - 16z^7 + 28z^5 - 19z^3 + 4z)a^{4} \\
 &\quad   + (3z^{10} - 12z^8 + 11z^6 - z^4 + 3z^2 - 2)a^{3}\\
&\quad + (z^{11} + 5z^9 - 42z^7 + 81z^5 - 66z^3 + 19z)a^{2}  \\
 &\quad  + (7z^{10} - 29z^8 + 42z^6 - 43z^4 + 32z^2 - 9)a^{1}\\
&\quad + (z^{11} + 7z^9 - 44z^7 + 79z^5 - 71z^3 + 23z)a^{0}  \\
 &\quad  + (4z^{10} - 13z^8 + 22z^6 - 38z^4 + 35z^2 - 11)a^{-1}\\
&\quad + (5z^9 - 17z^7 + 29z^5 - 28z^3 + 10z)a^{-2} + (3z^8 - 4z^6 + z^4 + 6z^2 - 4)a^{-3}\\
&\quad + (z^7 + 3z^5 - 3z^3 + 2z)a^{-4} + (4z^4 - 2z^2 - 1)a^{-5}  + z^3a^{-6} \\
\Lambda_{1} &=  (z^7 - 4z^5 + 4z^3)a^{5}\\
&\quad + (3z^8 - 13z^6 + 18z^4 - 9z^2)a^{4}  + (3z^9 - 9z^7 + 5z^5 - 2z^3 + 4z)a^{3} \\
&\quad + (z^{10} + 6z^8 - 35z^6 + 50z^4 - 30z^2 + 2)a^{2} \\
&\quad + (7z^9 - 22z^7 + 23z^5 - 23z^3 + 16z)a^{1}\\
&\quad + (z^{10} + 8z^8 - 38z^6 + 55z^4 - 34z^2 + 2)a^{0} \\
&\quad  + (4z^9 - 10z^7 + 14z^5 - 19z^3 + 16z)a^{-1} + (5z^8 - 16z^6 + 28z^4 - 16z^2 - 1)a^{-2}\\
&\quad + (2z^7 - z^3 + 4z)a^{-3} + (5z^4 - 3z^2 - 2)a^{-4}  + z^3a^{-5}\\
\Lambda_{D'} &=  (z^5 - 3z^3 + z)a^{4}+ (2z^6 - 7z^4 + 5z^2)a^{3} + (z^7 - 2z^5 - z)a^{2} \\
&\quad+ (3z^6 - 9z^4 + 7z^2)a^{1}  + (6z^3 - 1/z - 4z - 3z^5 + z^7)a^{0} \\
&\quad + (z^6 - 2z^4 + 3z^2 + 1)a^{-1}  + (3z^3 - 1/z - 2z)a^{-2} + z^2a^{-3}\\
V_{2n}(t) &= [- t^0 , t^{2n+9}]\\
V_{2}(t) &= -1+4t-7t^{2}+11t^{3}-14t^{4}+16t^{5}-14t^{6}+13t^{7}-10t^{8}+5t^{9}-3t^{10}+t^{11}\\
V_{D'}(t) &= t^{-1/2}-3t^{1/2}+2t^{3/2}-4t^{5/2}+3t^{7/2}-2t^{9/2}+2t^{11/2}-t^{13/2}
\end{aligned}
$$

\begin{figure}[h!]
  \centering
		\includegraphics[height=3cm]{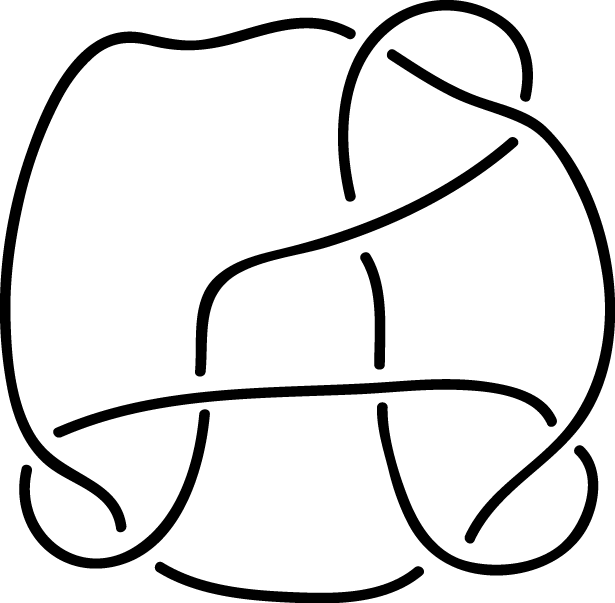} \hspace{10mm}
		\includegraphics[height=3cm]{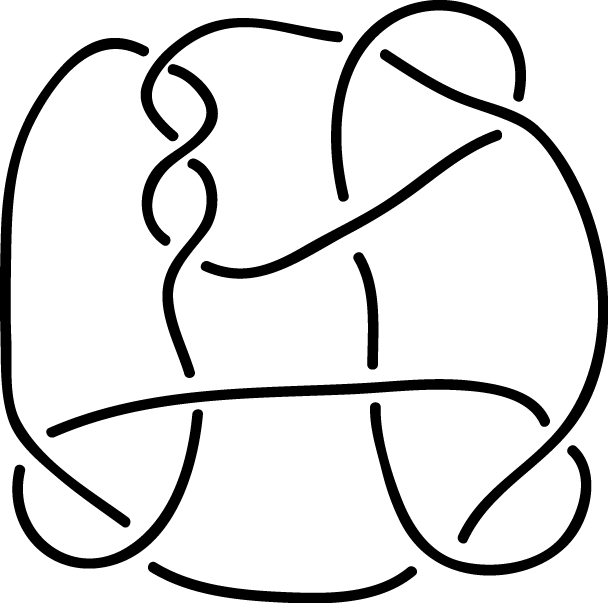}
  \caption{Diagrams $D'$ and $D_1$ for Theorem \ref{thm3_1}}
\end{figure}
\clearpage


\subsection*{Polynomials to prove Theorem \ref{thm3_2}.}

$$
\begin{aligned}
\Lambda_{2n} & = [z^2 a^{-(2n+5)} , 6 z a^4]\\
\Lambda_{2} &=  (z^9 - 7z^7 + 17z^5 - 17z^3 + 6z)a^{4}  + (2z^{10} - 13z^8 + 29z^6 - 29z^4 + 17z^2 - 6)a^{3} \\
&\quad + (z^{11} - z^9 - 21z^7 + 63z^5 - 61z^3 + 21z)a^{2} \\
&\quad+ (6z^{10} - 36z^8 + 75z^6 - 79z^4 + 56z^2 - 20)a^{1} \\
&\quad + (z^{11} + 3z^9 - 40z^7 + 93z^5 - 89z^3 + 34z)a^{0} \\
&\quad+ (4z^{10} - 20z^8 + 38z^6 - 51z^4 + 44z^2 - 17)a^{-1} \\
&\quad + (5z^9 - 24z^7 + 44z^5 - 48z^3 + 20z)a^{-2} + (3z^8 - 7z^6 - z^4 + 5z^2 - 4)a^{-3}\\
&\quad + (2z^7 - 3z^5)a^{-4} + (z^6 + z^2)a^{-5} + (3z^3 - z)a^{-6}  + z^2a^{-7}\\
\Lambda_{1} &= (z^8 - 6z^6 + 12z^4 - 9z^2 + 2)a^{4} \\
&\quad+ (2z^9 - 11z^7 + 20z^5 - 16z^3 + 8z - 2z^{-1})a^{3} \\
&\quad + (z^{10} - 20z^6 + 45z^4 - 33z^2 + 8)a^{2} \\
&\quad+ (6z^9 - 30z^7 + 51z^5 - 48z^3 + 30z - 7z^{-1})a^{1} \\
&\quad + (z^{10} + 4z^8 - 35z^6 + 62z^4 - 47z^2 + 13)a^{0} \\
&\quad+ (4z^9 - 16z^7 + 24z^5 - 31z^3 + 24z - 7z^{-1})a^{-1} \\
&\quad + (5z^8 - 20z^6 + 30z^4 - 26z^2 + 8)a^{-2}+ (3z^7 - 7z^5 + 4z^3 + z - 2z^{-1})a^{-3} \\
&\quad + (z^6 + z^4 - 2z^2 + 2)a^{-4} + (3z^3 - z^1)a^{-5}  + z^2a^{-6}\\
\Lambda_{D'} &= (z^5 - 4z^3 + 3z)a^{3}  + (z^6 - 4z^4 + 4z^2 - 2)a^{2} + (2z^5 - 7z^3 + 5z)a^{1} \\
&\quad+ (z^6 - 4z^4 + 6z^2 - 4)a^{0}+ (z^5 - 3z^3 + 3z)a^{-1}+ (2z^2 - 1)a^{-2}  + za^{-3}\\
V_{2n}(t) &= [t^{-(3n+8)} , - t^{-n+1}]\\
V_{2}(t) &= t^{-11}-3t^{-10}+4t^{-9}-7t^{-8}+8t^{-7}-9t^{-6}+10t^{-5}-7t^{-4}+7t^{-3}-4t^{-2} \\
&\quad+2t^{-1}-1\\
V_{D_1}(t) &= t^{-19/2}-3t^{-17/2}+5t^{-15/2}-7t^{-13/2}+8t^{-11/2}-9t^{-9/2}+7t^{-7/2}-7t^{-5/2} \\
&\quad+4t^{-3/2}-2t^{-1/2}+t^{1/2}
\end{aligned}
$$
\begin{figure}[h!]
  \centering
		\includegraphics[height=3cm]{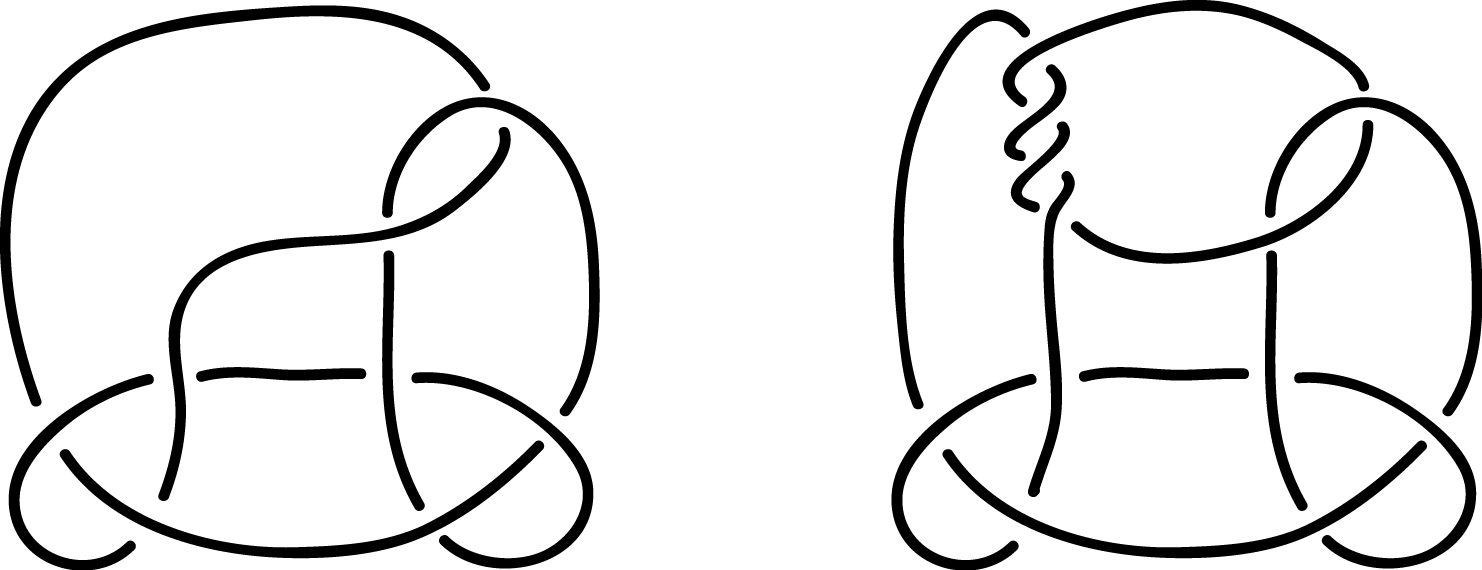}
  \caption{Diagrams $D'$ and $D_1$ for Theorem \ref{thm3_2}}
\end{figure}

\clearpage

\subsection*{Polynomials to prove Theorem \ref{thm3_3}.}

$$
\begin{aligned}
\Lambda_{2n} & = [z^3 a^{-(2n+5)} , -6 z^3 a^5]\\
\Lambda_{2} &= (z^9 - 6z^7 + 11z^5 - 6z^3)a^{5}+ (3z^{10} - 19z^8 + 41z^6 - 37z^4 + 13z^2)a^{4} \\ 
&\quad + (3z^{11} - 15z^9 + 20z^7 - 6z^5 + 5z^3 - 6z)a^{3} \\
&\quad+ (z^{12} + 4z^{10} - 48z^8 + 116z^6 - 116z^4 + 49z^2 - 2)a^{2} \\
&\quad + (7z^{11} - 36z^9 + 64z^7 - 66z^5 + 55z^3 - 25z)a^{1} \\
&\quad+ (z^{12} + 6z^{10} - 52z^8 + 118z^6 - 129z^4 + 58z^2 - 2)a^{0} \\
&\quad + (4z^{11} - 17z^9 + 34z^7 - 59z^5 + 59z^3 - 27z)a^{-1} \\
&\quad+ (5z^{10} - 21z^8 + 43z^6 - 56z^4 + 25z^2 + 1)a^{-2}  \\
&\quad+ (3z^9 - 3z^7 - 8z^5 + 14z^3 - 8z)a^{-3}+ (2z^8 - 2z^4 + z^2 + 1)a^{-4} \\
&\quad + (z^7 + 2z^5)a^{-5} + (4z^4 - 2z^2 - 1)a^{-6} + z^3a^{-7}\\
\Lambda_{1} &= (z^8 - 5z^6 + 7z^4 - 2z^2)a^{5}+ (3z^9 - 16z^7 + 28z^5 - 19z^3 + 4z)a^{4} \\
&\quad + (3z^{10} - 12z^8 + 11z^6 - z^4 + 3z^2 - 2)a^{3} \\
&\quad+ (z^{11} + 5z^9 - 42z^7 + 81z^5 - 66z^3 + 19z)a^{2} \\
&\quad + (7z^{10} - 29z^8 + 42z^6 - 43z^4 + 32z^2 - 9)a^{1} \\
&\quad+ (z^{11} + 7z^9 - 44z^7 + 79z^5 - 71z^3 + 23z)a^{0} \\
&\quad + (4z^{10} - 13z^8 + 22z^6 - 38z^4 + 35z^2 - 11)a^{-1} \\
&\quad+ (5z^9 - 17z^7 + 29z^5 - 28z^3 + 10z)a^{-2} + (3z^8 - 4z^6 + z^4 + 6z^2 - 4)a^{-3} \\
&\quad+ (z^7 + 3z^5 - 3z^3 + 2z)a^{-4} + (4z^4 - 2z^2 - 1)a^{-5}  + z^3a^{-6}\\
\Lambda_{D'} &= (z^5 - 3z^3 + z)a^{4}  + (2z^6 - 7z^4 + 5z^2)a^{3} + (z^7 - 2z^5 - z)a^{2} \\
&\quad+ (3z^6 - 9z^4 + 7z^2)a^{1} + (z^7 - 3z^5 + 6z^3 - 4z - 1z^{-1})a^{0} \\
&\quad+ (z^6 - 2z^4 + 3z^2 + 1)a^{-1} + (3z^3 - 2z - 1z^{-1})a^{-2}  + z^2a^{-3}\\
V_{2n}(t) &= [- t^{-3} , - t^{2n+7}]\\
V_{2}(t) &= -t^{-3}+4t^{-2}-6t^{-1}+11-14t+16t^{2}-18t^{3}+15t^{4}-13t^{5}+10t^{6}-5t^{7} \\
&\quad  +3t^{8}-t^{9}\\
V_{D'}(t) &= t^{-7/2}-3t^{-5/2}+2t^{-3/2}-4t^{-1/2}+3t^{1/2}-2t^{3/2}+2t^{5/2}-t^{7/2}
\end{aligned}
$$

\begin{figure}[h!]
  \centering
		\includegraphics[height=3cm]{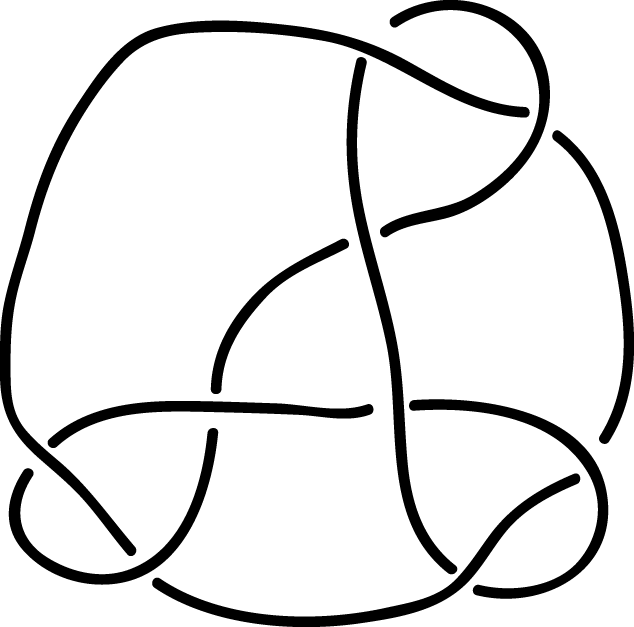} \hspace{10mm}
		\includegraphics[height=3cm]{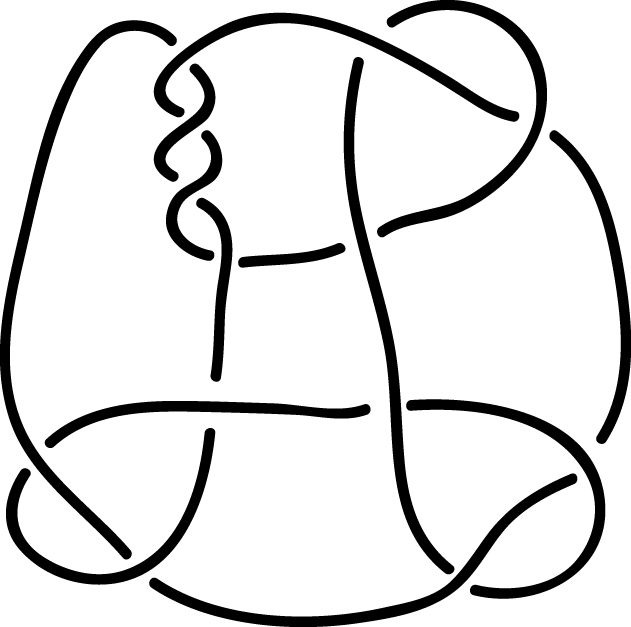}
  \caption{Diagrams $D'$ and $D_1$ for Theorem \ref{thm3_3}}
\end{figure}

\clearpage

\subsection*{Polynomials to prove Theorem \ref{thm3_4}.}

$$
\begin{aligned}
\Lambda_{2n} &= [- z^2 a^{-(2n+4)} , -z^2 a^6] \\
\Lambda_{2} &= (z^8 - 4z^6 + 4z^4 - z^2)a^{6}+ (4z^9 - 18z^7 + 24z^5 - 11z^3 + 2z)a^{5} \\
&\quad + (6z^{10} - 25z^8 + 30z^6 - 14z^4 + 5z^2 - 1)a^{4} \\
&\quad+ (4z^{11} - 6z^9 - 28z^7 + 61z^5 - 38z^3 + 7z)a^{3} \\
&\quad + (z^{12} + 12z^{10} - 67z^8 + 112z^6 - 91z^4 + 34z^2 - 5)a^{2} \\
&\quad+ (8z^{11} - 22z^9 + 3z^7 + 28z^5 - 28z^3 + 7z)a^{1} \\
&\quad + (z^{12} + 11z^{10} - 58z^8 + 110z^6 - 112z^4 + 53z^2 - 9)a^{0} \\
&\quad+ (4z^{11} - 9z^9 + 9z^7 - 7z^5 + 3z^3 + z)a^{-1} \\
&\quad + (5z^{10} - 16z^8 + 35z^6 - 43z^4 + 29z^2 - 6)a^{-2}+ (3z^9 - 4z^7 + 6z^5 - z)a^{-3} \\
&\quad + (z^8 + 3z^6 - 3z^4 + 3z^2 - 2)a^{-4} + (4z^5 - 4z^3)a^{-5} \\
&\quad+ (z^4 - z^2)a^{-6}\\
\Lambda_{1} &=(z^7 - 3z^5 + 2z^3)a^{6}+ (4z^8 - 14z^6 + 14z^4 - 4z^2)a^{5} \\
&\quad + (6z^9 - 19z^7 + 17z^5 - 8z^3 + 2z)a^{4}+ (4z^{10} - 2z^8 - 26z^6 + 34z^4 - 13z^2)a^{3} \\
&\quad + (z^{11} + 13z^9 - 53z^7 + 66z^5 - 42z^3 + 10z)a^{2} \\
&\quad+ (8z^{10} - 14z^8 - 10z^6 + 26z^4 - 10z^2)a^{1} \\
&\quad + (z^{11}  + 12z^9 - 49z^7 + 76z^5 - 54z^3 + 14z)a^{0}+ (4z^{10} - 6z^9 + z^6 + 9z^4 - 2z^2)a^{-1} \\
&\quad + (5z^9 - 16z^7 + 35z^5 - 28z^3 + 6z + 1z^{-1})a^{-2} \\
&\quad+ (2z^8 - z^6 + 4z^4 - 2z^2 - 1)a^{-3} + (5z^5 - 6z^3 + 1z^{-1})a^{-4} \\
&\quad+ (z^4 - z^2)a^{-5}\\
\Lambda_{D'} &=(z^5 - 2z^3)a^{5} + (3z^6 - 8z^4 + 3z^2)a^{4} + (3z^7 - 8z^5 + 5z^3 - z)a^{3} \\
&\quad+ (z^8 + z^6 - 6z^4 + 3z^2)a^{2} + (4z^7 - 11z^5 + 12z^3 - 3z)a^{1} \\
&\quad+ (z^8 - 2z^6 + 5z^4 - 3z^2 + 1)a^{0} + (z^7 - 2z^5 + 6z^3 - 3z)a^{-1} + (3z^4 - 3z^2)a^{-2} \\
&\quad+ (z^3 - z)a^{-3}
\end{aligned}
$$

\begin{figure}[b!]
  \centering
		\includegraphics[height=3cm]{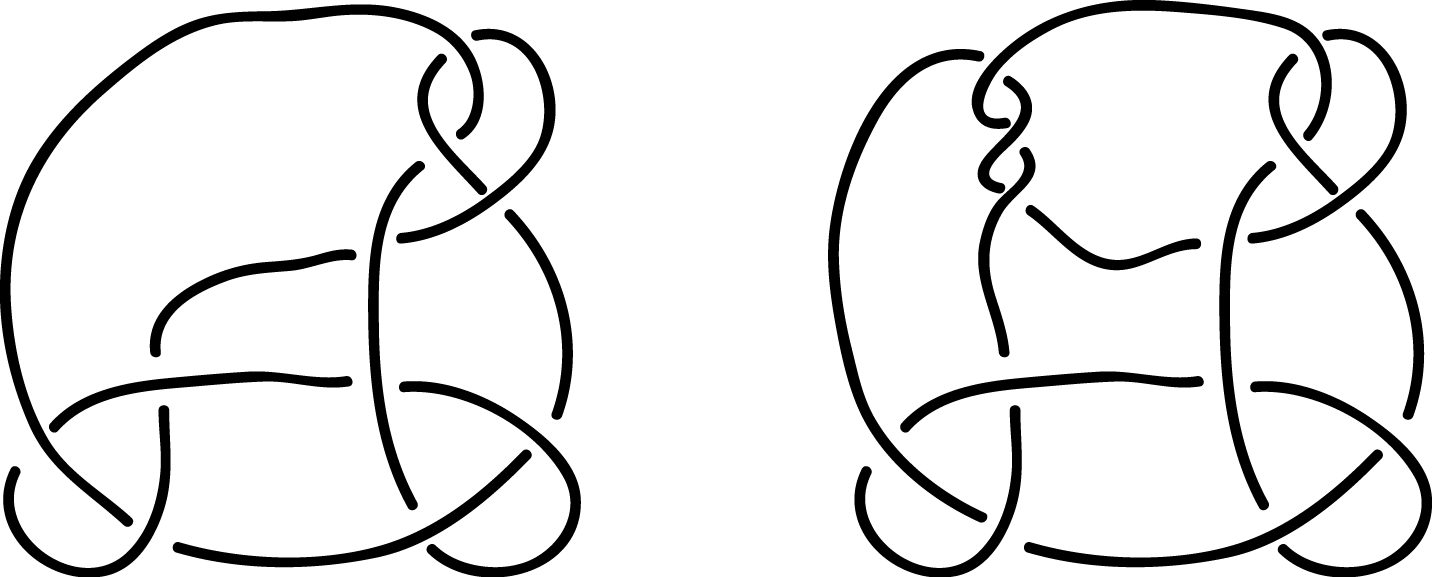}
  \caption{Diagrams $D'$ and $D_1$ for Theorem \ref{thm3_4}}
\end{figure}

$$
\begin{aligned}
V_{2n}(t) &= [t^{-(3n+1)} , t^{-n+9}]\\
V_{2}(t) &= t^{-4}-4t^{-3}+9t^{-2}-15t^{-1}+22-26t+27t^{2}-25t^{3}+21t^{4}-14t^{5}+8t^{6}-4t^{7} \\
&\quad+t^{8}\\
V_{D_1}(t) &= t^{-5/2}-5t^{-3/2}+9t^{-1/2}-16t^{1/2}+21t^{3/2}-23t^{5/2}+22t^{7/2}-20t^{9/2}+14t^{11/2} \\
&\quad-8t^{13/2}+4t^{15/2}-t^{17/2}
\end{aligned}
$$

\clearpage
\subsection*{Polynomials to prove Theorem \ref{thm3_5}.}
$$
\begin{aligned}
\Lambda_{2n} &= [- z^4 a^{-(2n+6)} , -4a^4] \\
\Lambda_{2} &= (6z^6 - 19z^4 + 17z^2 - 4)a^{4}+ (3z^9 - 3z^7 - 5z^5 + 3z^3 + z)a^{3} \\
&\quad + (7z^{10} - 22z^8 + 44z^6 - 58z^4 + 32z^2 - 6)a^{2} \\
&\quad+ (5z^{11} - 2z^9 - 13z^7 + 18z^5 - 10z^3 + z)a^{1} \\
&\quad + (z^{12} + 21z^{10} - 68z^8 + 100z^6 - 84z^4 + 36z^2 - 7)a^{0} \\
&\quad+ (10z^{11} + 4z^9 - 60z^7 + 86z^5 - 44z^3 + 7z)a^{-1} \\
&\quad + (z^{12} + 25z^{10} - 58z^8 + 51z^6 - 30z^4 + 19z^2 - 5)a^{-2} \\
&\quad+ (5z^{11} + 23z^9 - 76z^7 + 77z^5 - 34z^3 + 7z)a^{-3} \\
&\quad + (11z^{10} - z^8 - 32z^6 + 27z^4 - 4z^2 - 1)a^{-4}+ (14z^9 - 21z^7 + 6z^5)a^{-5} \\
&\quad + (11z^8 - 20z^6 + 11z^4 - 2z^2)a^{-6} + (5z^7 - 8z^5 + 3z^3)a^{-7}
+ (z^6 - z^4)a^{-8}\\
\Lambda_{1} &=(6z^5 - 13z^3 + 7z)a^{3} \\
&\quad+ (3z^8 - 5z^4 + 2z^2)a^{2} + (7z^9 - 15z^7 + 27z^5 - 26z^3 + 5z + 1z^{-1})a^{1} \\
&\quad+ (5z^{10} + 3z^8 - 17z^6 + 16z^4 - 4z^2 - 1)a^{0} \\
&\quad + (z^{11} + 22z^9 - 55z^7 + 58z^5 - 30z^3 + 3z + 1z^{-1})a^{-1} \\
&\quad+ (10z^{10} + 9z^8 - 63z^6 + 66z^4 - 19z^2)a^{-2} \\
&\quad + (z^{11} + 25z^9 - 54z^7 + 34z^5 - 10z^3 + 3z)a^{-3} \\
&\quad+ (5z^{10} + 19z^8 - 67z^6 + 59z^4 - 16z^2)a^{-4} \\
&\quad + (10z^9 - 9z^7 - 12z^5 + 11z^3 - 2z)a^{-5} \\
&\quad+ (10z^8 - 20z^6 + 13z^4 - 3z^2)a^{-6} + (5z^7 - 9z^5 + 4z^3)a^{-7} 
+(z^6 - z^4)a^{-8}\\
\Lambda_{D'} &=(z^5 - 2z^3 + z)a^{5} + (3z^6 - 5z^4 + 3z^2 - 1)a^{4} \\
&\quad + (5z^7 - 8z^5 + 5z^3 - 2z)a^{3}+ (4z^8 - 9z^4 + 7z^2 - 2)a^{2} \\
&\quad + (z^9 + 13z^7 - 30z^5 + 23z^3 - 6z)a^{1}+ (8z^8 - 4z^6 - 14z^4 + 12z^2 - 2)a^{0} \\
&\quad + (z^9 + 14z^7 - 32z^5 + 22z^3 - 4z)a^{-1}+ (4z^8 + 3z^6 - 17z^4 + 11z^2 - 2)a^{-2} \\
&\quad + (6z^7 - 10z^5 + 5z^3 - z)a^{-3} + (4z^6 - 7z^4 + 3z^2)a^{-4} + (z^5 - z^3)a^{-5}
\\
V_{2n}(t) &= [t^{-3n-12}, 3t^{-n-2}]\\
V_{2}(t) &= t^{-15}-5t^{-14}+12t^{-13}-22t^{-12}+33t^{-11}-43t^{-10}+47t^{-9}-45t^{-8}+39t^{-7} \\
&\quad-28t^{-6}+17t^{-5}-8t^{-4}+3t^{-3}\\
V_{D_1}(t)& = t^{-27/2}-5t^{-25/2}+11t^{-23/2}-19t^{-21/2}+28t^{-19/2}-34t^{-17/2}+34t^{-15/2} \\
&\quad-32t^{-13/2}+25t^{-11/2}-16t^{-9/2}+8t^{-7/2}-3t^{-5/2}
\end{aligned}
$$

\begin{figure}[b!]
  \centering
		\includegraphics[height=3cm]{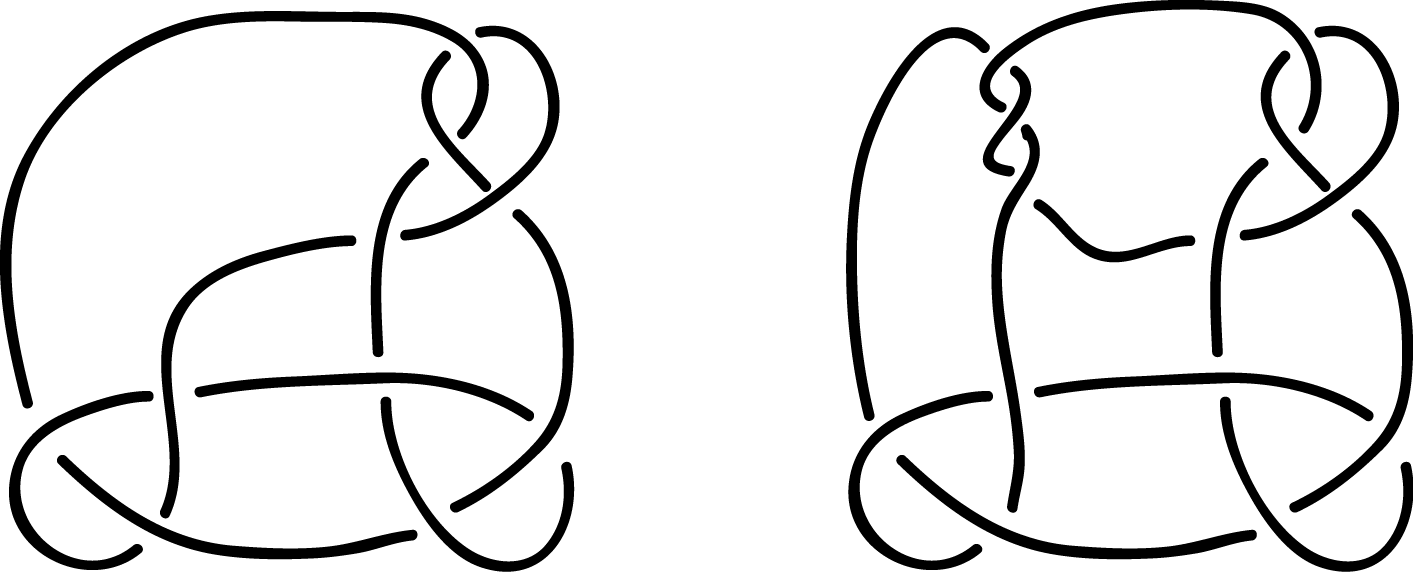}
  \caption{Diagrams $D'$ and $D_1$ for Theorem \ref{thm3_5}}
\end{figure}

\subsection*{Polynomials to prove Theorem \ref{thm3_6}.}

$$
\begin{aligned}
\Lambda_{2n} &= [z^2 a^{-(2n+6)}, (-1)^n (2n+4) a^4] \\
\Lambda_{2} &= (z^{10} - 8z^8 + 24z^6 - 34z^4 + 23z^2 - 6)a^{4}+  (2z^{11} - 14z^9 + 34z^7 - 34z^5 + 12z^3)a^{3} \\
&\quad + (z^{12} - 33z^8 + 110z^6 - 137z^4 + 73z^2 - 14)a^{2} \\
&\quad+  (7z^{11} - 42z^9 + 78z^7 - 46z^5 + z^3 + 3z)a^{1}  \\
&\quad+ (z^{12} + 8z^{10} - 77z^8 + 184z^6 - 183z^4 + 86z^2 - 16)a^{0} \\
&\quad+  (5z^{11} - 20z^9 + 7z^7 + 35z^5 - 33z^3 + 8z)a^{-1}  \\
&\quad+ (9z^{10} - 47z^8 + 83z^6 - 78z^4 + 42z^2 - 9)a^{-2} \\
&\quad+  (8z^9 - 34z^7 + 41z^5 - 24z^3 + 6z)a^{-3} + (5z^8 - 14z^6 + 3z^4 + 4z^2 - 2)a^{-4} \\
&\quad+  (3z^7 - 6z^5 + z^3)a^{-5} + (z^6 + z^4 - z^2)a^{-6} + (3z^3 - z)a^{-7} \\
&\quad+  (z^2)a^{-8}\\
\Lambda_{1} &= (z^9 - 7z^7 + 18z^5 - 21z^3 + 11z - 2z^{-1})a^{4} \\
&\quad+  (2z^{10}- 12z^8+ 24z^6- 18z^4+ 3z^2+ 1)a^{3}  \\
&\quad+ (z^{11}+ z^9- 31z^7+ 82z^5- 81z^3+ 33z - 5z^{-1})a^{2} \\
&\quad+  (7z^{10} - 35z^8 + 50z^6 - 19z^4 - 2z^2 + 1)a^{1}  \\
&\quad+ (z^{11} + 9z^9 - 67z^7 + 126z^5 - 99z^3 + 38z - 6z^{-1})a^{0} \\
&\quad+  (5z^{10} - 15z^8 - 5z^6 + 29z^4 - 16z^2 + 3)a^{-1}  \\
&\quad+ (9z^9 - 39z^7 + 53z^5 - 40z^3 + 19z - 4z^{-1})a^{-2} \\
&\quad+  (8z^8 - 30z^6 + 32z^4 - 16z^2 + 3)a^{-3} + (4z^7 - 9z^5 + 2z^3 + 2z - 1z^{-1})a^{-4} \\
&\quad+  (z^6 + 2z^4 - 4z^2 + 1)a^{-5} + (3z^3 - z)a^{-6} \\
&\quad+  (z^2)a^{-7}\\
\Lambda_{D'} &= (z^6 - 5z^4 + 7z^2 - 3)a^{3} +  (z^7 - 4z^5 + 3z^3)a^{2} + (3z^6 - 13z^4 + 14z^2 - 4)a^{1} \\
&\quad+  (z^7 - 3z^5 + z^3)a^{0} + (2z^6 - 8z^4 + 9z^2 - 3)a^{-1}  \\
&\quad+  (z^5 - 2z^3 + z)a^{-2} + (2z^2 - 1)a^{-3} \\
&\quad+  (z)a^{-4}\\
V_{2n}(t) &= [t^{-3n-8} , t^{-n+2}]\\
V_{2}(t) &= t^{-11}-3t^{-10}+5t^{-9}-8t^{-8}+11t^{-7}-13t^{-6}+14t^{-5}-13t^{-4}+11t^{-3} \\
&\quad-8t^{-2}+5t^{-1}-2+t.\\
V_{D_1}(t)& = t^{-19/2}-3t^{-17/2}+6t^{-15/2}-9t^{-13/2}+11t^{-11/2}-13t^{-9/2}+12t^{-7/2} \\
&\quad-11t^{-5/2}+8t^{-3/2}-5t^{-1/2}+2t^{1/2}-t^{3/2}
\end{aligned}
$$

\begin{figure}[!ht]
  \centering
		\includegraphics[height=3cm]{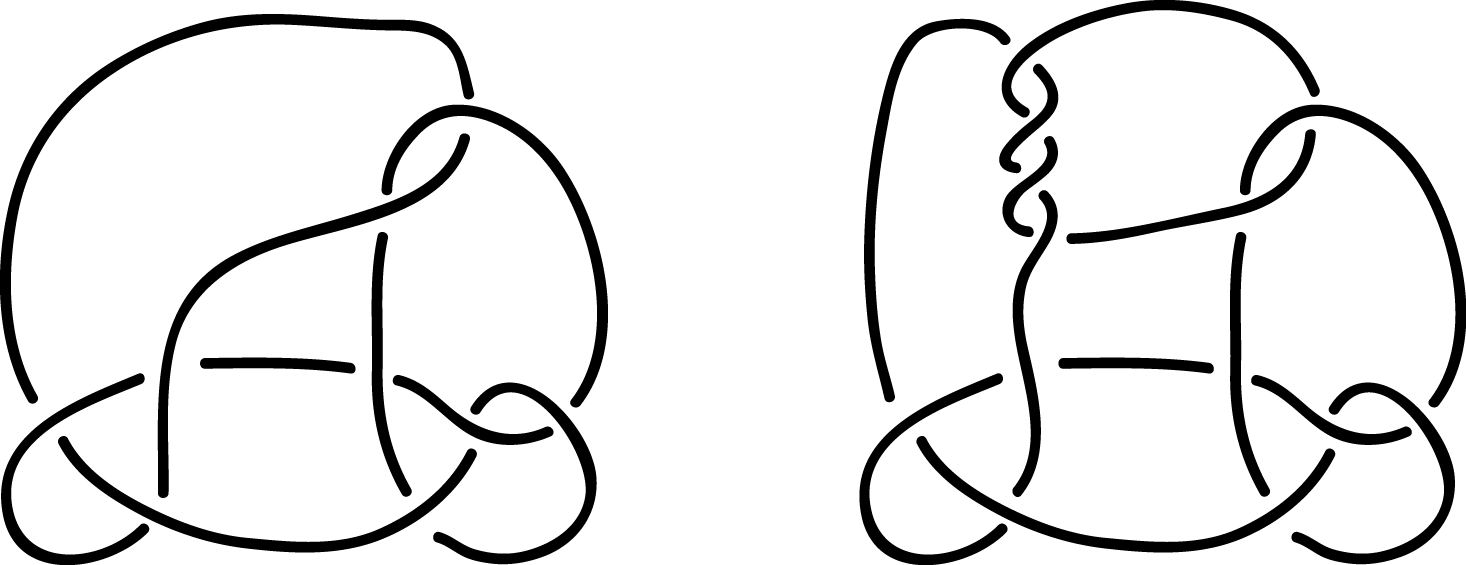}
  \caption{Diagrams $D'$ and $D_1$ for Theorem \ref{thm3_6}}
\end{figure}

\clearpage

\subsection*{Polynomials to prove Theorem \ref{thm3_7}.}

$$
\begin{aligned}
\Lambda_{2n} &= [z^3 a^{-(2n+5)}, (-1)^{n+1} 2 za^5] \\
\Lambda_{2} &= (z^9 - 6z^7 + 12z^5 - 9z^3 + 2z)a^{5}+  (3z^{10} - 18z^8 + 38z^6 - 35z^4 + 14z^2 - 2)a^{4} \\
&\quad + (3z^{11} - 12z^9 + 6z^7 + 20z^5 - 21z^3 + 5z)a^{3} \\
&\quad+  (z^{12} + 7z^{10} - 58z^8 + 122z^6 - 112z^4 + 50z^2 - 9)a^{2} \\
&\quad + (8z^{11} - 30z^9 + 20z^7 + 17z^5 - 17z^3 + 5z)a^{1} \\
&\quad+  (z^{12} + 13z^{10} - 77z^8 + 137z^6 - 125z^4 + 61z^2 - 11)a^{0} \\
&\quad + (5z^{11} - 9z^9 - 17z^7 + 32z^5 - 17z^3 + 3z)a^{-1} \\
&\quad+  (9z^{10} - 33z^8 + 49z^6 - 53z^4 + 30z^2 - 7)a^{-2} \\
&\quad + (8z^9 - 24z^7 + 28z^5 - 18z^3 + 3z)a^{-3}+  (4z^8 - 4z^6 - z^4 + 4z^2 - 2)a^{-4} \\
&\quad + (z^7 + 5z^5 - 5z^3 + 2z)a^{-5} + (4z^4 - z^2)a^{-6} \\
&\quad+  (z^3)a^{-7}\\
\Lambda_{1} &= (z^8 - 5z^6 + 8z^4 - 4z^2)a^{5} \\
&\quad+  (3z^9 - 15z^7 + 26z^5 - 19z^3 + 5z)a^{4} + (3z^{10}- 9z^8 + 14z^4 - 7z^2)a^{3} \\
&\quad+  (z^{11} + 8z^9 - 49z^7 + 79z^5 - 56z^3 + 17z - 1z^{-1})a^{2} \\
&\quad + (8z^{10} - 22z^8 + 2z^6 + 18z^4 - 7z^2 + 1)a^{1} \\
&\quad+  (z^{11} + 14z^9 - 65z^7 + 91z^5 - 64z^3 + 19z - 1z^{-1})a^{0} \\
&\quad + (5z^{10} - 5z^8 - 21z^6 + 30z^4 - 11z^2)a^{-1} \\
&\quad+  (9z^9 - 29z^7 + 41z^5 - 31z^3 + 8z)a^{-2} + (7z^8 - 18z^6 + 23z^4 - 9z^2)a^{-3} \\
&\quad+  (2z^7 + 3z^5 - 3z^3 + z)a^{-4} + (5z^4 - 2z^2)a^{-5} \\
&\quad+  (z^3)a^{-6}\\
\Lambda_{D'} &= (z^6 - 4z^4 + 4z^2 - 1)a^{4} \\
&\quad+  (2z^7 - 8z^5 + 8z^3 - 2z)a^{3} + (z^8 - z^6 - 6z^4 + 6z^2 - 1)a^{2} \\
&\quad+  (4z^7 - 14z^5 + 13z^3 - 4z)a^{1} + (z^8 - z^6 - 2z^4 + z^2 - 1)a^{0} \\
&\quad+  (2z^7 - 6z^5 + 8z^3 - 3z)a^{-1} + z^6a^{-2} + (3z^3 - z)a^{-3} 
+ (z^2)a^{-4}\\
V_{2n}(t) &= [-t^{-3n-6}, - t^{-n+4}]\\
V_{2}(t) &= -t^{-9}+4t^{-8}-9t^{-7}+14t^{-6}-19t^{-5}+23t^{-4}-23t^{-3}+21t^{-2}-17t^{-1}\\
&\quad+12-6t+3t^{2}-t^{3}\\
V_{D_1}(t) &= -t^{-15/2}+5t^{-13/2}-10t^{-11/2}+14t^{-9/2}-19t^{-7/2}+20t^{-5/2}-19t^{-3/2}\\
&\quad +16t^{-1/2}-12t^{1/2}+6t^{3/2}-3t^{5/2}+t^{7/2}
\end{aligned}
$$

\begin{figure}[h!]
  \centering
		\includegraphics[height=3cm]{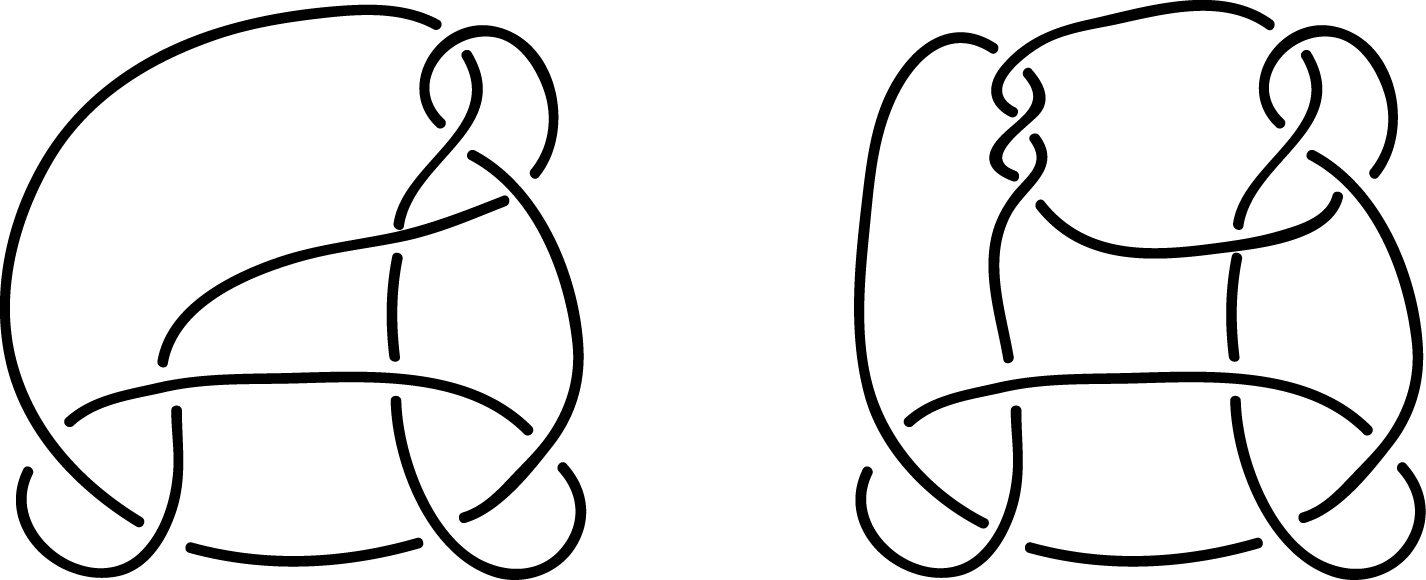}
  \caption{Diagrams $D'$ and $D_1$ for Theorem \ref{thm3_7}}
\end{figure}

\section*{Acknowledgments}
This work was supported  in part by the National Research Foundation of Korea Grant
funded by the Korean Government (NRF-2013-056086).

\end{document}